\documentclass[11pt,a4paper,english]{article}
\usepackage[latin1]{inputenc}
\usepackage{babel}
\usepackage[centertags]{amsmath}
\usepackage{amsfonts}
\usepackage{amssymb}
\usepackage{color}
\usepackage{amsthm}
\usepackage{newlfont}
\usepackage{graphicx}
\usepackage{dsfont}
\usepackage{geometry}
\usepackage{mathrsfs}
\usepackage{authblk}
\usepackage{verbatim}
\usepackage{hyperref}

\newcommand{\prob}{\ensuremath{\mathsf{P}}}
\newcommand{\beq}{\begin{equation}}
\newcommand{\media}{\ensuremath{\mathsf{E}}}
\newcommand{\eeq}{\end{equation}}

\newcommand{\reali}{\ensuremath{\mathbb{R}}}

\newcommand{\cI}{\mathcal{I}}
\newcommand{\cC}{\mathcal{C}}
\newcommand{\cS}{\mathcal{S}}

\newcommand{\cF}{\mathcal{F}}
\newcommand{\cL}{\mathcal{L}}

\newcommand{\cO}{\mathcal{O}}

\newcommand{\pprob}{\widetilde{\mathsf{P}}}
\newcommand{\mmedia}{\widetilde{\mathsf{E}}}

\newcommand{\eps}{\varepsilon}

\makeatletter
\def\namedlabel#1#2{\begingroup
   \def\@currentlabel{#2}
   \label{#1}\endgroup}
\makeatother
\DeclareMathOperator*{\esssup}{ess\,sup}

\newtheorem{theorem}{Theorem}[section]
\newtheorem{lemma}[theorem]{Lemma}
\newtheorem{corollary}[theorem]{Corollary}
\newtheorem{proposition}[theorem]{Proposition}
\newtheorem{remark}[theorem]{Remark}
\newtheorem{assumption}[theorem]{Assumption}

\makeatletter  % so @ will be a letter
\@addtoreset{equation}{section}

\makeatother  % so  @ will not be a letter
\begin{document}
\title{\textbf{On Lipschitz continuous optimal stopping boundaries}}
\author{Tiziano De Angelis\thanks{School of Mathematics, University of Leeds, Woodhouse Lane LS2 9JT, Leeds, United Kingdom. \texttt{t.deangelis@leeds.ac.uk}}\qquad \&\qquad Gabriele Stabile\thanks{Dipartimento di Metodi e Modelli per l'Economia, il Territorio e la Finanza, Sapienza University of Rome, Via del Castro Laurenziano 9, 00161, Rome, Italy. \texttt{gabriele.stabile@uniroma1.it}}
}
\date{\today}
\maketitle

\textbf{Abstract.} We obtain a probabilistic proof of the local Lipschitz continuity for the optimal stopping boundary of a class of problems with state space $[0,T]\times\reali^d$, $d\ge 1$. To the best of our knowledge this is the only existing proof that relies exclusively upon stochastic calculus, all the other proofs making use of PDE techniques and integral equations. Thanks to our approach we obtain our result for a class of diffusions whose associated second order differential operator is not necessarily uniformly elliptic. The latter condition is normally assumed in the related PDE literature.

\medskip

{\textbf{Keywords}}: optimal stopping, free boundary problems, Lipschitz free boundaries.

\smallskip

{\textbf{MSC2010 subject classification}}: 60G40, 35R35.
%\smallskip

\section{Introduction}

In this work we deal with optimal stopping problems of the form
\begin{align}\label{V0}
v(t,x)=\sup_{0\le\tau\le T-t}\media\left[\int_0^\tau h(t+s,X^x_s)ds+\mathds{1}_{\{\tau<T-t\}}f(t+\tau,X^x_\tau)+\mathds{1}_{\{\tau=T-t\}}g(X^x_{\tau})\right]
\end{align}
where $\media$ denotes the expectation operator. For $d\ge 1$ and $d'\ge1$, given a suitable $\reali^d$-valued function $\mu$ and a $d\times d'$ matrix $\sigma$, the process $X\in\reali^d$ follows the dynamic 
\begin{align*}%\label{X}
X^{x}_t=x+\int_0^t\mu(X^{x}_s)ds+\sigma B_t,\quad t\ge 0,
\end{align*}
with $B$ a $\reali^{d'}$-valued Brownian motion. The main focus of our study is the analysis of the regularity of the optimal stopping boundary, i.e.~the boundary of the set in $[0,T)\times\reali^d$ where $v=f$. 

Under mild assumptions on $\mu$, $f$, $g$ and $h$ we provide a probabilistic representation of the gradient of $v$. The latter is used, along with more technical requirements on $f$, $g$ and $h$, to prove that the optimal stopping boundary may be expressed in terms of a locally Lipschitz continuous function $b:[0,T]\times\reali^{d-1}\to\reali$. One of the main features in our work is that we do not assume uniform non-degeneracy of the diffusion so that standard results based on PDE theory cannot be easily applied. 

It is well known that optimal stopping theory goes hand in hand with the theory of free boundary problems in PDE and the question of regularity of optimal stopping boundaries (free boundaries) has been the object of intensive study. The one dimensional case $d=1$ attracted the interest of several mathematicians who developed approaches ranging from probability to analysis.
Early contributions to the topic were made in \cite{Kot73}, \cite{McK65} and \cite{VM}, among others. In \cite{Kot73} and \cite{VM} it was proven that the free-boundary $b$ is differentiable in the open interval $(0,T)$ for a certain class of problems involving one-dimensional Brownian motion or solutions of one-dimensional SDEs with regular coefficents. Other papers employing PDE methods are for example \cite{CDH67}, \cite{Sch76}, where infinite differentiability of the free boundary in the Stefan problem is proved, and \cite{Fri75} where $C^1$ regularity of the boundary is obtained for a certain class of variational problems. The study of the optimal boundary of the American put option is perhaps one of the most renowned examples in this field and for an overview of existing results one may refer to \cite{BX09}, \cite{Bl06}, \cite{CC07}, \cite{CCJZ08}, \cite{Eks04}, \cite{J91}, \cite{Kim} \cite{LV03}, and \cite{Pe05} among others. Finally it is worth recalling that a thorough discussion of analytical methods for free boundary problems on $[0,T]\times\reali$ related to the heat operator may be found in the monograph \cite{Can} (see also \cite[Ch.~8]{Fri08}). In the latter, as well as in several of the above references, the first step in the analysis of the regularity of the free boundary is to prove that it is Lipschitz continuous or at least H\"older continuous with constant $\alpha>1/2$. 

There is also a large body of literature addressing similar questions in higher dimensions. Accounting in full for these results is a difficult task and it falls outside the reach of our work. However for our purposes it is interesting to recall the following fact: Lipschitz regularity for the free boundary of certain Stefan problems (with $d\ge 1$) can be upgraded to $C^{1,\alpha}$ regularity for some $\alpha\in(0,1)$ and eventually to $C^\infty$ regularity, under suitable technical conditions. Detailed derivations of this informal statement may be found in the monographs \cite{CS} and \cite{PSU} and references therein (see also \cite{LauSal09} for the study of American options written on several assets and with convex payoff). 

In the literature on optimal stopping the vast majority of papers studying problems of the form \eqref{V0} with $d=1$ addresses the question of continuity of the boundary without looking at higher regularity (of course with the exception of the works mentioned above; see \cite{DeA15} for some results and further references). Moreover, even the question of continuity becomes difficult to handle for $d>1$. In the case of $d=2$ and $T=\infty$, specific examples were addressed in \cite{DeAFF} and \cite{PJ}, while a more complete answer was recently provided in \cite{Pe17}. 

Notably, Shreve and Soner \cite{SS91b, SS91} address a problem of singular stochastic control which is equivalent to one of optimal stopping of the form \eqref{V0}, and characterise the optimal boundary as a real-valued, Lipschitz continuous function on $[0,T]\times\reali^{d-1}$, $d\ge 1$. In their work they employ the equivalence between the problem of singular control and the one of optimal stopping, and study the latter purely by means of PDE methods similar to those in \cite{Ben-Lio82}. Regularity of the free boundary is used in \cite{SS91b, SS91} to obtain a classical solution to a variational problem with gradient constraint related to the singular control problem. It is worth mentioning that the same authors had previously shown $C^2$ regularity for the optimal boundary of a two-dimensional singular control problem on an infinite time horizon \cite{SS89}. However in the setting of \cite{SS89} we are not aware of any direct link to an optimal stopping problem and therefore it is harder to draw a parallel with our work.  

From the above discussion we learn that a reasonable attempt towards the study of regularity for optimal boundaries in optimal stopping theory should start form establishing their Lipschitz continuity. Of course this can be achieved in several instances by the PDE methods illustrated in the references above but instead we aim at finding a fully probabilistic approach. Under assumptions similar to those adopted in \cite{SS91b, SS91}, our work not only serves the purpose of bridging the PDE literature and the probabilistic one but it also contributes new results.

One of our main contributions is to prove that for $d>1$ local Lipschitz continuity of the optimal boundary can be obtained without requiring uniform ellipticity of the operator $\sigma\sigma^\top$ (see Theorems \ref{thm:F} and \ref{thm:G}, and Example 2 in Section \ref{sec:examples}). Relaxing this requirement makes it difficult to apply standard PDE results (including \cite{CS} and \cite{PSU}) and the methods used in \cite{SS91b, SS91} are no longer valid. In the special case $d=1$ (see Theorem \ref{thm:lip-d1}) we are able to localize the assumptions made in \cite{SS91b, SS91} and in particular the one relative to the running cost, i.e.~our function $h$. 
Such relaxation allows us to apply our results to a wider class of examples than the one previously covered. For instance we can apply them in problems of irreversible capacity expansions where the running profit is expressed by a Cob-Douglas-type production function (see, e.g.~\cite{Ch-Haus09} and Example 1 in Section \ref{sec:examples}).
A more detailed comparison between our setting and the one in \cite{SS91b,SS91} is provided in Remark \ref{rem:comp}. 

We also notice that our functional \eqref{V0} allows a rather generic time-space dependence of the functions $f$, $g$ and $h$, while at the same time the dynamic of $X$ allows state dependent drifts and correlations between the driving noises (i.e.~$\sigma$ is not necessarily diagonal). For $d=1$ a generic time dependence of $f$ and $h$ makes it extremely hard and often impossible to establish monotonicity of the optimal boundary as a function of time. The latter is normally a key feature in the study of the boundary's continuity. One advantage of our approach is that instead we do not need such monotonicity to establish Lipschitz continuity (see~\cite{DeAS17} for a recent application in actuarial context). Moreover if the boundary is Lipschitz then $v\in C^1([0,T)\times\reali)$ (see Remark \ref{rem:C1}).

Our method consists of two main steps which we can formally summarise as follows. In the first step we find a probabilistic representation of the time/space derivatives of the value function. The latter is then used in the second step along with the implicit function theorem to obtain bounds on the gradient of the optimal boundary. Notice that, while the second step is somehow in line with ideas in \cite{SS91}, the first step is entirely new.

It is important to remark that despite the technical assumptions that we make, one of the main contributions of our work is the methodology. As it is often the case in optimal stopping and free boundary problems, in order to be able to give general results, one has to impose fairly strong conditions on the problem data. However, when considering specific examples it is possible to find ways around the technicalities and still apply the same methods. This is indeed true also for the theory that we are developing here and in Section \ref{sec:examples} we provide some examples of such extensions.

The rest of the paper is organised as follows. In Section \ref{sec:setup} we provide a rigorous formulation of the problem outlined in \eqref{V0} along with the standing assumptions. In Section \ref{sec:value} we obtain a probabilistic representation formula for the gradient $\nabla_x v$ and for bounds on the time derivative $\partial_t v$ (see Theorem \ref{thm:lip}). Some other technical estimates are performed before passing to Section \ref{sec:lip}. In the latter we finally give our main results regarding existence of a locally-Lipschitz continuous optimal boundary for problem \eqref{V0}. This result is given under three different sets of assumptions: in Theorem \ref{thm:lip-d1} for $d=1$ and in Theorems \ref{thm:F} and \ref{thm:G} for $d\ge 2$. In Section \ref{sec:examples} we show some applications of our results and their extensions in specific examples.

\section{Setup and problem formulation}\label{sec:setup}
Consider a complete probability space $(\Omega,\mathcal {F}, \prob)$ equipped with the natural filtration $\mathbb{F} := (\mathcal {F}_t)_{t \geq 0}$ generated by a $\reali^{d'}$-valued Brownian motion $(B_t)_{t \geq 0}$. Assume that $\mathbb{F}$ is completed with $\prob$-null sets and let $X\in\reali^d$ evolve according to 
\begin{align}\label{X}
X^{x}_t=x+\int_0^t\mu(X^{x}_s)ds+\sigma B_t,\quad t\ge 0,
\end{align}
where $\mu\in C^1(\reali^d;\reali^d)$ with sub-linear growth and $\sigma$ is a $d\times d'$ matrix. 
We denote by $\langle\cdot\,,\cdot\rangle$ the scalar product in $\reali^d$ and by $\|\cdot\|_d$ the Euclidean norm in $\reali^d$. Notice that $\sigma\sigma^\top$ is assumed to be non-negative but not necessarily uniformly elliptic. This means that it may exist $\xi\in\reali^d$ such that $\langle \sigma\sigma^\top \xi,\xi\rangle=0$.

Throughout the paper we will often use $\prob_{t,x}(\,\cdot\,)=\prob(\,\cdot\,|X_t=x)$ and $\prob_x=\prob_{0,x}$, so that $\media_{t,x} f(X_s)=\media f(X^{t,x}_s)$, $s\ge t$, for any function $f$ which is Borel-measurable and integrable.  With no loss of generality we will assume $\Omega= C([0,T];\reali^d)$ so that $t\mapsto \omega(t)$ is the canonical process and $\theta_\cdot$ the shifting operator such that $\theta_s \omega(t)=\omega(t+s)$.

For $T\in(0,+\infty)$ we consider optimal stopping problems of the form
\begin{align}\label{V}
v(t,x)=\sup_{0\le\tau\le T-t}\media\left[\int_0^\tau h(t+s,X^x_s)ds+\mathds{1}_{\{\tau<T-t\}}f(t+\tau,X^x_\tau)+\mathds{1}_{\{\tau=T-t\}}g(X^x_{\tau})\right]
\end{align}
where $f$, $g$ and $h$ are real-valued with $f\in C^{1,2}([0,T]\times\reali^d)$, $h\in C^{1,1}([0,T]\times\reali^d)$ and $g\in C^{2}(\reali^d)$. 
In the infinite horizon case, i.e.~$T=+\infty$, we consider
\begin{align}
v(t,x)=\sup_{\tau\ge 0}\media\left[\int_0^\tau h(t+s,X^x_s)ds+f(t+\tau,X^x_\tau)\right]
\end{align}
with $f$ and $h$ as above and, according to \cite[Ch.~3]{Shir}, we set
\begin{align*}
\mathds{1}_{\{\tau=+\infty\}}f(t+\tau,X^x_\tau):=\limsup_{s\to\infty}f(s,X^x_s),\quad\text{$\prob$-a.s.}
\end{align*}

In what follows conditions at $T$ for the terminal value $g(X_T)$ are understood to hold only for $T<+\infty$ and can always be neglected for $T=+\infty$. From now on we assume that for all $(t,x)\in[0,T]\times\reali^d$ it holds 
\begin{align}
\label{integr}\media\Big[\int_0^{T-t} |h(t+s,X^x_s)|ds+|g(X^x_{T-t})|+\sup_{0\le s\le T-t}|f(t+s,X^x_s)|\Big]<+\infty.
\end{align}
Moreover, if $T=+\infty$ then we also assume 
\begin{align}
\label{integr2}\limsup_{s\to\infty}f(t+s,X^x_s)=\lim_{s\to\infty}f(t+s,X^x_s)=0,\quad\text{$\prob$-a.s.}
\end{align}
Both assumptions are fulfilled in the examples of Section \ref{sec:examples}.
\begin{remark}
Notice that the dynamic \eqref{X} and the optimisation problem \eqref{V} are general enough to include for example models involving geometric Brownian motion and Ornstein-Uhlenbeck.  
\end{remark}

To avoid further technicalities we also assume that $v$ is a lower semi-continuous function. Often such regularity (or even continuity) is easy to check in specific examples (e.g., those in Section \ref{sec:examples}). There also exist mild sufficient conditions that guarantee lower semi-continuity of $v$ in more general settings (see for instance \cite[Ch.~3]{Shir}. See also Remark 2.10 and eq.~(2.2.80) in \cite[Ch.I, Sec.~2]{PS}). 

The continuation set $\cC$ and the stopping set $\cS$ are given by
\begin{align}
&\cC:=\{(t,x)\in[0,T)\times\reali^d\,:\,v(t,x)>f(t,x)\}\\
&\cS:=\{(t,x)\in[0,T)\times\reali^d\,:\,v(t,x)=f(t,x)\}\cup(\{T\}\times\reali).
\end{align}
From standard optimal stopping theory we know that, in our setting, \eqref{integr} and lower semi-continuity of $v$ are sufficient for the optimality of
\begin{align}
\label{entry}
\tau_*(t,x)= \inf \left \{ s \in [0,T-t] \! :  \! (t+s,X_s^x) \in \mathcal{S} \right \}
\end{align}
provided that $f(T,x)\le g(x)$, if $T<+\infty$ (see \cite[Ch.~I, Sec.~2, Cor.~2.9]{PS}). For the infinite horizon case notice that if $\prob_{t,x}(\tau_*<+\infty)<1$, then there is no optimal stopping time and $\tau_*$ is a (optimal) Markov time (according to the terminology in \cite[Ch.~3, Thm.~3]{Shir}). However methods used in the next sections work for both finite and infinite values of $\tau_*$ thanks to \eqref{integr2}. 

For arbitrary $(t,x)\in[0,T]\times\reali^d$ let
\begin{align*}
Y_s:= v(t+s,X^x_s)+\int_0^s h(t+u,X^x_u)du.
\end{align*}
Since $v$ is lower semi-continuous and using standard results in optimal stopping (see \cite[Ch.~I, Sec.~2, Thm.~2.4]{PS}) we have that $(Y_{s})_{0\le s\le T-t}$ is $\prob$-a.s.~right-continuous and 
\begin{align}
\label{supmg}&Y_s\quad\text{is a supermartingale for $s\in[0,T-t]$,}\\
\label{mg}&Y_{s\wedge\tau_*}\quad\text{is a martingale for $s\in[0,T-t]$.}
\end{align}
Notice in particular that since $Y$ is right-continuous then the process $s\mapsto v(t+s,X^x_s)$ is $\prob$-a.s.~right continuous as well. 
As a note of caution we remark that if $T=+\infty$ then \eqref{supmg} continues to hold on $[0,+\infty]$ because $Y$ is a uniformly integrable super-martingale thanks to \eqref{integr} (see, e.g.~\cite[Thm.~9, Ch.~1]{Shir}). Instead, \eqref{mg} only holds on $[0,\infty)$.

We denote by $\cL$ the infinitesimal generator associated to $X$ and in particular we have
\begin{align}\label{def:L}
\cL F(x)=\frac{1}{2}\sum^d_{i,j=1}(\sigma\sigma^\top)_{i,j}\frac{\partial^2F}{\partial x_i\partial x_j}(x)+\sum^d_{i=1}\mu_i(x)\frac{\partial F}{\partial x_i}(x),\qquad F\in C^2(\reali^d;\reali).
\end{align}
For future frequent use we also introduce the following notation
\begin{align}
\label{def:m} m(t,x):=\left(\partial_tf+\cL f\right)(t,x)\quad\text{and}\quad n(x):=\cL g(x).
\end{align}

Since $\mu\in C^1(\reali^d;\reali^d)$ then the flow $x\mapsto X^x$ is differentiable (\cite{Pr}, Chapter V.7). 
Here we denote the initial point in \eqref{X} by $x=(x_1,\ldots,x_d)$, the $i$-th component of $X^x$ by $X^{x,i}$, the partial derivative with respect to $x_k$ by $\partial_k=\tfrac{\partial}{\partial x_k}$, and the derivative of $X^x$ with respect to the initial point $x_k$ by $\partial_k X^x=(\partial_k X^{x,1},\ldots \partial_k X^{x,d})$. We define the process $\partial X^x$ as a $d\times d$ matrix with entries $\partial_k X^{x,j}$ for $j,k=1,\ldots d$ and the maps $t\mapsto\partial_k X^{x,j}_t$ are $\prob$-a.s.~continuous with dynamics given by 
\begin{align}\label{DX}
\partial_k X^{x,j}_t=&\delta_{j,k}+\int_0^t\sum^d_{\ell=1}\partial_{\ell}\mu_j(X^x_s)\partial_kX^{x,\ell}_s ds=\delta_{j,k}+\int_0^t\langle \nabla_x\mu_j(X^x_s),\partial_kX^{x}_s\rangle ds.
\end{align}
In what follows we also assume that for any compact $K\subset\reali^d$ it holds
\begin{align}\label{grad-L2}
\sup_{x\in K}\media\left[\sup_{0\le t\le T}\|\partial_k X^x_t\|_d^2\right]<+\infty \quad\text{for all $k=1,\ldots d$}.
\end{align}

The next will be a standing assumption throughout the paper
\begin{assumption}[\textbf{Regularity $f,g,h$}.]\label{ass:fgh}
For any compact $K\subset[0,T]\times\reali^d$, there exists $c_K>0$ such that for all $(t,x)\in K$ we have
\begin{align*}
&\media\left[\int_0^{T-t} \|\nabla_x h(t+s,X^x_s)\|_d^2\, ds+\sup_{0\le s\le T-t}\|\nabla_x f(t+s,X^x_s)\|_d^2+\|\nabla_x g(X^x_{T-t})\|_d^2\right]\le c_K,\\[+5pt]
&\media\left[\int_0^{T-t} \!\! |\partial_t h(t+s,X^x_s)|^2 ds +\sup_{0\le s\le T-t}|\partial_t f(t+s,X^x_s)|^2+ |h(T,X^x_{T-t})+n(X^x_{T-t})|^2\right]\le c_K.
\end{align*}
%Moreover the bounds are uniform for $(t,x)$ in compact subsets of $[0,T]\times\reali^d$.
\end{assumption}
\vspace{+4pt}

\noindent It is important to remark that Assumption \ref{ass:fgh} and \eqref{grad-L2} are used in Theorem \ref{thm:lip} in order to pass limits under expectations, thanks to uniform integrability. Therefore, in specific examples one can expect to find weaker sufficient conditions that allow this step in the proof.
 
\section{Properties of the value function}\label{sec:value}

In this section we provide useful bounds for the gradient of the value function $v$ and some other technical results. These are obtained by often using the following condition
\begin{description}
\item [(A$_1$) Terminal value.\namedlabel{A$_1$}{(A$_1$)}] If $T<+\infty$ we have $g(x)\ge f(T,x)$.
\end{description}
Before stating the next theorem it is useful to introduce the functions
\begin{align}\label{upvt}
\overline{v}(t,x)= \media\Big[&\int_0^{\tau_*}\partial_t h(t+s,X^x_s)ds\\
&+\mathds{1}_{\{\tau_*<T-t\}} \partial_t f(t+\tau_*,X^x_{\tau_*})-\mathds{1}_{\{\tau_*=T-t\}}\left(h(T,X^x_{T-t})+n(X^x_{T-t})\right)\Big]\nonumber
\end{align}
and
\begin{align}\label{downvt}
\underline{v}(t,x)= \media\Big[&\int_0^{\tau_*}\partial_t h(t+s,X^x_s)ds+\mathds{1}_{\{\tau_*<T-t\}} \partial_t f(t+\tau_*,X^x_{\tau_*})\\
&-\mathds{1}_{\{\tau_*=T-t\}}\Big(|h(T,X^x_{T-t})+n(X^x_{T-t})|+|\partial_t f(T,X^{x}_{T-t})|\Big)\Big].\nonumber
\end{align}

\begin{theorem}\label{thm:lip}
Assume condition \ref{A$_1$}. Then the value function $v$ is locally Lipschitz continuous on $[0,T]\times\reali^d$ and for a.e.~$(t,x)$ we have
\begin{align}\label{grad-v}
\hspace{-10pt}
\partial_k v(t,x)=&\media\Big[\int_0^{\tau_*}\langle\nabla_x h(t+s,X^x_s),\partial_k X^x_s\rangle ds\\
&\hspace{+5pt}+\mathds{1}_{\{\tau_*<T-t\}}\langle \nabla_x f(t+\tau_*,X^x_{\tau_*}),\partial_k X^x_{\tau_*}\rangle+\mathds{1}_{\{\tau_*=T-t\}}\langle \nabla_x g(X^x_{T-t}),\partial_k X^x_{T-t}\rangle\Big]\nonumber
\end{align}
and 
\begin{align}\label{vt}
\underline{v}(t,x)\le \partial_t v(t,x)\le \overline{v}(t,x).
\end{align}
\end{theorem}
\begin{proof}
\emph{Step 1. (Spatial derivative).} Here we show that $v(t,\cdot)$ is locally Lipschitz and \eqref{grad-v} holds for a.e.~$x\in\reali^d$ and each given $t\in[0,T]$ (notice that the null set where $v(t,\cdot)$ is not differentiable may a priori depend on $t$). First we obtain bounds for the left and right derivative of $v(t,\cdot)$.

Fix $(t,x)\in[0,T]\times\reali^d$ and take $\eps>0$. For an arbitrary $k$ we denote for simplicity $x_\eps=(x_1,\ldots x_k+\eps,\ldots x_d)$, and consider the processes $X^{x_\eps}=(X^{x_\eps,1},\ldots \ldots X^{x_\eps,d})$ and $X^x=(X^{x,1},\ldots X^{x,d})$. We remark that all components of the vector process $X^{x_\eps}$ are affected by the shift in the initial point.

We denote by $\tau=\tau_*(t,x)$ the optimal stopping time (independent of $\eps$) for the problem with initial data $(t,x)$. Using such optimality we first obtain
\begin{align*}
v(t,x_\eps)&-v(t,x)\\
\ge& \media\left[\int_0^\tau\left(h(t+s,X^{x_\eps}_s)-h(t+s,X^{x}_s)\right)ds+\mathds{1}_{\{\tau<T-t\}}\left(f(t+\tau,X^{x_\eps}_\tau)-f(t+\tau,X^{x}_\tau)\right)\right]\\
&+\media\left[\mathds{1}_{\{\tau=T-t\}}\left(g(X^{x_\eps}_{T-t})-g(X^x_{T-t})\right)\right].
\end{align*}
Dividing both sides of the above expression by $\eps$ and recalling Assumption \ref{ass:fgh} and \eqref{grad-L2} we can pass to the limit as $\eps\to0$ and use dominated convergence to conclude that
\begin{align}\label{grad-1}
\liminf_{\eps\to0}&\frac{v(t,x_\eps)-v(t,x)}{\eps}\nonumber\\
\ge& \media\left[\int_0^\tau\langle\nabla_x h(t+s,X^x_s),\partial_k X^x_s\rangle ds+\mathds{1}_{\{\tau<T-t\}}\langle \nabla_x f(t+\tau,X^x_\tau),\partial_k X^x_\tau \rangle\right]\\
&+\media\left[\mathds{1}_{\{\tau=T-t\}}\langle \nabla_x g(X^x_{T-t}),\partial_k X^x_{T-t}\rangle\right].\nonumber
\end{align}

To obtain a reverse inequality, pick $\delta>0$ and denote $x_\delta=(x_1,\ldots x_k-\delta,\ldots x_d)$ and $X^{x_\delta}=(X^{x_\delta,1},\ldots X^{x_\delta,d})$. Since $\tau$ is optimal in $v(t,x)$ and sub-optimal in $v(t,x_\delta)$ we have
\begin{align*}
v(t,x)&-v(t,x_\delta)\\
\le& \media\left[\int_0^\tau\left(h(t+s,X^{x}_s)-h(t+s,X^{x_\delta}_s)\right)ds+\mathds{1}_{\{\tau<T-t\}}\left(f(t+\tau,X^{x}_\tau)-f(t+\tau,X^{x_\delta}_\tau)\right)\right]\\
&+\media\left[\mathds{1}_{\{\tau=T-t\}}\left(g(X^{x}_{T-t})-g(X^{x_\delta}_{T-t})\right)\right].
\end{align*}
Dividing both sides by $\delta$, taking limits and using dominated convergence again we obtain
\begin{align}\label{grad-2}
\limsup_{\delta\to0}&\frac{v(t,x)-v(t,x_\delta)}{\delta}\nonumber\\
\le& \media\left[\int_0^\tau\langle\nabla_x h(t+s,X^x_s),\partial_k X^x_s\rangle ds+\mathds{1}_{\{\tau<T-t\}}\langle \nabla_x f(t+\tau,X^x_\tau),\partial_k X^x_\tau \rangle\right]\\
&+\media\left[\mathds{1}_{\{\tau=T-t\}}\langle \nabla_x g(X^x_{T-t}),\partial_k X^x_{T-t}\rangle\right].\nonumber
\end{align}

Now, \eqref{grad-1} gives a lower bound for the right derivative with respect to $x_k$ whereas \eqref{grad-2} provides an upper bound for the corresponding left derivative. If $x$ is a point of differentiability of $v(t,\cdot)$ then \eqref{grad-1} and \eqref{grad-2} imply that \eqref{grad-v} holds at that point. It remains to show that $v(t,\cdot)$ is locally Lipschitz so that a.e.~$x\in\reali^d$ is a point of differentiability.

With the same notation as above let $\tau_\eps=\tau_*(t,x_\eps)$ be optimal for the problem with initial data $(t,x_\eps)$. By analogous arguments to those used previously and using Assumption \ref{ass:fgh} and \eqref{grad-L2} we find
\begin{align}\label{lipx}
v&(t,x_\eps)-v(t,x)\nonumber\\
\le& \media\left[\int_0^{\tau_\eps}\left(h(t+s,X^{x_\eps}_s)-h(t+s,X^{x}_s)\right)ds+\mathds{1}_{\{\tau_\eps<T-t\}}\left(f(t+\tau_\eps,X^{x_\eps}_{\tau_\eps})-f(t+\tau_\eps,X^{x}_{\tau_\eps})\right)\right]\nonumber\\
&+\media\left[\mathds{1}_{\{\tau_\eps=T-t\}}\left(g(X^{x_\eps}_{T-t})-g(X^x_{T-t})\right)\right]\le c(t,x)\eps,
\end{align}
for some $c(t,x)>0$, which is uniform for $(t,x)$ on a compact. Notice also that for the last inequality we have used 
\begin{align*}
\left\|X^{x_\eps}_{\tau_\eps}-X^{x}_{\tau_\eps}\right\|_d\le\eps\cdot \sum_k\sup_{0\le s\le T}\left\|\partial_k X^z_s\right\|_d
\end{align*}
by the mean value theorem, with suitable $z\in\reali^d$ such that $\|z-x\|_d\le \eps$. 

The estimate in \eqref{lipx} and \eqref{grad-1} imply $|v(t,x_\eps)-v(t,x)|\le \hat c(t,x) \eps$, for some other constant $\hat c(t,x)>0$ which can be taken uniform over compact sets. Symmetric arguments allow to prove also that $|v(t,x_\delta)-v(t,x)|\le \hat c(t,x) \delta$, with $x_\delta$ as in \eqref{grad-2}.
\vspace{+5pt}

\emph{Step 2. (Time derivative).} Here we show that $t\mapsto v(t,x)$ is locally Lipschitz and \eqref{vt} holds for a.e.~$t\in[0,T]$ and each given $x\in\reali^d$. We start by providing bounds for the left and right derivatives of $v(\,\cdot\,,x)$.

Fix $(t,x)\in[0,T]\times\reali^d$ and let $\eps>0$. Then letting $\tau=\tau_*(t,x)$ optimal for the problem with initial data $(t,x)$ we notice that $\tau$ is admissible for the problem with initial data $(t-\eps,x)$. Using \eqref{supmg} and \eqref{mg} we obtain the following upper bound.
\begin{align}\label{grad-5}
v(t,x)-v(t-\eps,x)\le\media\Big[&\int_0^\tau\left(h(t+s,X^x_s)-h(t-\eps+s,X^x_s)\right)ds\nonumber\\
&+v(t+\tau,X^x_\tau)-v(t-\eps+\tau,X^x_\tau)\Big]. 
\end{align}
Now we notice that since $v\ge f$ on $[0,T]\times\reali^d$ and $v=f$ in $\cS$, by right continuity of $v(t+\cdot,X^x_\cdot)$ one has
\begin{align*}
&v(t+\tau,X^x_\tau)-v(t-\eps+\tau,X^x_\tau)\le f(t+\tau,X^x_\tau)-f(t-\eps+\tau,X^x_\tau)\quad\text{on $\{\tau<T-t\}$}\\
&v(t+\tau,X^x_\tau)-v(t-\eps+\tau,X^x_\tau)\le g(X^x_{T-t})-v(T-\eps,X^x_{T-t})\quad\text{on $\{\tau=T-t\}$}.
\end{align*}
Moreover from \eqref{V} we also have
\begin{align*}
v(T-\eps,X^x_{T-t})\ge& \media_{X^x_{T-t}}\left[\int_0^\eps h(T-\eps+s, X_s)ds+ g(X_\eps)\right]\nonumber\\
=&g(X^x_{T-t})+\media_{X^x_{T-t}}\left[\int_0^\eps \left(h(T-\eps+s, X_s)+n(X_s)\right)ds\right].
\end{align*}
Collecting the above estimates and using the mean value theorem we conclude
\begin{align}\label{eq:new1}
\frac{1}{\eps}&(v(t,x)-v(t-\eps,x))\nonumber\\
\le& \media\left[\int_0^\tau\partial_t h(t-\eps'_s+s,X^x_s)ds+\mathds{1}_{\{\tau<T-t\}}\partial_t f(t-\eps''_\tau+\tau,X^x_\tau)\right]\\
&-\media_x\left[\mathds{1}_{\{\tau=T-t\}}\media_{X_{T-t}}\left[\frac{1}{\eps}\int_0^\eps \left(h(T-\eps+s, X_s)+n(X_s)\right)ds\right]\right]\nonumber 
\end{align}
for $\eps'_s$ and $\eps''_\tau$ in $[0,\eps]$. Letting $\eps\to0$ and using Assumption \ref{ass:fgh} we get
\begin{align}\label{grad-3}
\limsup_{\eps\to0}&\frac{v(t,x)-v(t-\eps,x)}{\eps}\\
\le &\media\Big[\int_0^{\tau}\partial_t h(t+s,X^x_s)ds+\mathds{1}_{\{\tau<T-t\}} \partial_t f(t+\tau,X^x_{\tau})\Big]\nonumber\\
&-\media\Big[\mathds{1}_{\{\tau=T-t\}}\left(h(T,X^x_{T-t})+n(X^x_{T-t})\right)\Big].\nonumber
\end{align}
To prove a reverse inequality we notice that $\tau\wedge(T-t-\eps)$ is admissible for the problem with initial data $(t+\eps,x)$, so that by using \eqref{supmg} and \eqref{mg} and arguing as above we obtain
\begin{align}\label{vt0}
v(&t+\eps,x)-v(t,x)\nonumber\\
\ge&\media\Big[\int_0^{\tau\wedge(T-t-\eps)}\left(h(t+\eps+s,X^x_s)-h(t+s,X^x_s)\right)ds-\mathds{1}_{\{\tau> T-t-\eps\}}\int^\tau_{T-t-\eps}h(t+s,X^x_s)ds\Big]\nonumber\\
&+\media\Big[\mathds{1}_{\{\tau\le T-t-\eps\}}\left(f(t+\eps+\tau,X^x_\tau)-f(t+\tau,X^x_\tau)\right)\Big]\nonumber\\
&+\media\Big[\mathds{1}_{\{\tau> T-t-\eps\}}\left(g(X^x_{T-t-\eps})-v(t+\tau,X^x_\tau)\right)\Big].
\end{align}
We can collect the two terms with the indicator of $\{\tau>T-t-\eps\}$, use iterated conditioning and the martingale property \eqref{mg} to get 
\begin{align*}
\media&\Big[\mathds{1}_{\{\tau> T-t-\eps\}}\left(g(X^x_{T-t-\eps})-v(t+\tau,X^x_\tau)-\int^\tau_{T-t-\eps}h(t+s,X^x_s)ds\right)\Big]\\
=&\media\Big[\mathds{1}_{\{\tau> T-t-\eps\}}\left(g(X^x_{T-t-\eps})-\media\left[ v(t+\tau,X^x_\tau)+\int^\tau_{T-t-\eps}h(t+s,X^x_s)ds\Big|\cF_{T-t-\eps}\right]\right)\Big]\\
=&\media\Big[\mathds{1}_{\{\tau> T-t-\eps\}}\left(g(X^x_{T-t-\eps})-v(T-\eps,X^x_{T-t-\eps})\right)\Big].
\end{align*}
To estimate the last term we argue as follows  
\begin{align*}
v&(T-\eps,X^x_{T-t-\eps})\\
=&\esssup_{0\le\sigma\le\eps} \media_{X^x_{T-t-\eps}}\left[\int_{0}^{\sigma}h(T-\eps+s,X_s)ds+\mathds{1}_{\{\sigma<\eps\}}f(T-\eps+\sigma,X_{\sigma})+\mathds{1}_{\{\sigma=\eps\}}g(X_{\eps})\right]\\
=&\esssup_{0\le\sigma\le\eps} \media_{X^x_{T-t-\eps}}\left[\int_{0}^{\sigma}h(T-\eps+s,X_s)ds+g(X_{\sigma})+\mathds{1}_{\{\sigma<\eps\}}\left(f(T-\eps+\sigma,X_{\sigma})-g(X_\sigma)\right)\right]
\end{align*}
\begin{align*}
=&g(X^x_{T-t-\eps})\\
&+\esssup_{0\le\sigma\le\eps} \media_{X^x_{T-t-\eps}}\Big[\int_{0}^{\sigma}\left(h(T-\eps+s,X_s)+n(X_s)\right)ds+\mathds{1}_{\{\sigma<\eps\}}\Big(f(T,X_{\sigma})-g(X_\sigma)\Big)\\
&\phantom{+\esssup_{0\le\sigma\le\eps} \media_{X^x_{T-t-\eps}}+}
-\mathds{1}_{\{\sigma<\eps\}}\int_{T-\eps+\sigma}^{T}\partial_t f(u,X_\sigma)du\Big].
\end{align*}
Using that $f(T,x)\le g(x)$ by condition \ref{A$_1$}, we get  
\begin{align}\label{eq:gab0}
v(&T-\eps,X^x_{T-t-\eps})\nonumber\\[+3pt]
\le&g(X^x_{T-t-\eps})\\
&+\esssup_{0\le\sigma\le\eps} \media_{X^x_{T-t-\eps}}\Big[\int_{0}^{\sigma}\!\!\left(h(T-\eps+s,X_s)\!+\!n(X_s)\right)ds+\!\int_{T-\eps+\sigma}^{T}\!|\partial_t f(u,X_\sigma)|du\Big]\nonumber\\
\le&g(X^x_{T-t-\eps})+\media_{X^x_{T-t-\eps}}\Big[\int_{0}^{\eps}|h(T-\eps+s,X_s)+n(X_s)|ds\Big]\notag\\
&+\media_{X^x_{T-t-\eps}}\Big[\int_{0}^{\eps}\sup_{r\le \eps}|\partial_tf(T-\eps+s\wedge(\eps-r)+r,X_r)|ds\Big].\nonumber
\end{align}

Plugging the estimates above inside \eqref{vt0} we then obtain
\begin{align*}
\frac{1}{\eps}&(v(t+\eps,x)-v(t,x))\\
\ge &\media\left[\int_0^{\tau\wedge(T-t-\eps)}\partial_t h(t+\eps'_s+s,X^x_s)ds+\mathds{1}_{\{\tau\le T-t-\eps\}}\partial_t f(t+\eps''_\tau+\tau,X^x_\tau)\right]\\
&-\media_x\left[\mathds{1}_{\{\tau>T-\eps-t\}}\media_{X_{T-t-\eps}}\Big[\frac{1}{\eps}\int_{0}^{\eps}|h(T-\eps+s,X_s)+n(X_s)|ds\Big]\right]\\
&-\media_x\left[\mathds{1}_{\{\tau>T-\eps-t\}}\media_{X_{T-t-\eps}}\Big[\frac{1}{\eps}\int_{0}^{\eps}\sup_{r\le \eps}|\partial_t f(T-\eps+s\wedge(\eps-r)+r,X_r)|ds\Big]\right]
\end{align*}
for suitable $\eps'_s$ and $\eps''_\tau$ in $[0,\eps]$. Taking limits as $\eps\to 0$ we conclude
\begin{align}\label{grad-4}
\liminf_{\eps\to0}&\frac{v(t+\eps,x)-v(t,x)}{\eps}\\
\ge &\media\Big[\int_0^{\tau}\partial_t h(t+s,X^x_s)ds+\mathds{1}_{\{\tau<T-t\}} \partial_t f(t+\tau,X^x_{\tau})\Big]\nonumber\\
&-\media\Big[\mathds{1}_{\{\tau=T-t\}}\left(|h(T,X^x_{T-t})+n(X^x_{T-t})|+|\partial_t f(T,X^x_{T-t})|\right)\Big].\nonumber
\end{align}

So far we have established a lower bound for the right-derivative and an upper bound for the left-derivative of $v(\cdot,x)$. Hence \eqref{vt} holds for all $t\in[0,T]$ at which $v(\cdot,x)$ is differentiable, thanks to \eqref{grad-3} and \eqref{grad-4}. Next we prove that $v(\,\cdot\,,x)$ is indeed locally Lipschitz so that \eqref{vt} holds for a.e.~$t\in[0,T]$ and each given $x\in\reali^d$. 

Let us set $\tau_\eps:=\tau_*(t-\eps,x)$ and notice that $\tau_\eps\wedge(T-t)$ is admissible for the problem with initial data $(t,x)$. Therefore arguing as in \eqref{vt0} we get
\begin{align}\label{vt01}
v(&t,x)-v(t-\eps,x)\nonumber\\
\ge&\media\Big[\int_0^{\tau_\eps\wedge(T-t)}\left(h(t+s,X^x_s)-h(t-\eps+s,X^x_s)\right)ds-\mathds{1}_{\{\tau_\eps> T-t\}}\int^{\tau_\eps}_{T-t}h(t-\eps+s,X^x_s)ds\Big]\nonumber\\
&+\media\Big[\mathds{1}_{\{\tau_\eps\le T-t\}}\left(f(t+\tau_\eps,X^x_{\tau_\eps})-f(t-\eps+\tau_\eps,X^x_{\tau_\eps})\right)\Big]\nonumber\\
&+\media\Big[\mathds{1}_{\{\tau_\eps> T-t\}}\left(g(X^x_{T-t})-v(t-\eps+\tau_\eps,X^x_{\tau_\eps})\right)\Big].
\end{align}
Repeating step by step the arguments that follow \eqref{vt0} we obtain
\begin{align*}
\frac{1}{\eps}&(v(t,x)-v(t-\eps,x))\\
\ge &\media\left[\int_0^{\tau_\eps\wedge(T-t)}\partial_t h(t-\eps'_s+s,X^x_s)ds+\mathds{1}_{\{\tau_\eps\le T-t\}}\partial_t f(t-\eps''_{\tau_\eps}+\tau_\eps,X^x_{\tau_\eps})\right]\\
&-\media_x\left[\mathds{1}_{\{\tau_\eps>T-t\}}\media_{X_{T-t}}\Big[\frac{1}{\eps}\int_{0}^{\eps}|h(T-\eps+s,X_s)+n(X_s)|ds\Big]\right]\\
&-\media_x\left[\mathds{1}_{\{\tau_\eps>T-t\}}\media_{X_{T-t}}\Big[\frac{1}{\eps}\int_{0}^{\eps}\sup_{r\le \eps}|\partial_tf(T-\eps+s\wedge(\eps-r)+r,X_r)|ds\Big]\right],
\end{align*}
with $\eps'_s$ and $\eps''_{\tau_\eps}$ in $[0,\eps]$.
Using Assumption \ref{ass:fgh} and the above expression it is clear that we can find $c(t,x)>0$, which is uniform for $(t,x)$ in a compact and such that $v(t,x)-v(t-\eps,x)\ge -c(t,x)\,\eps$. The latter, together with \eqref{grad-3} imply that $|v(t,x)-v(t-\eps,x)|\le \hat c(t,x)\eps$ for some other $\hat c(t,x)>0$ which is uniform on compact sets. A symmetric argument can be used to obtain an analogous bound for $|v(t+\eps,x)-v(t,x)|$ and therefore $v(\,\cdot\,,x)$ is indeed locally Lipschitz.
\vspace{+5pt}

\emph{Step 3. (Lipschitz property).} To complete the proof it only remains to observe that, combining results in step 1 and 2 above, we obtain that $v$ is locally Lipschitz on $[0,T]\times\reali^d$. Hence, it is differentiable for a.e.~$(t,x)\in[0,T]\times\reali^d$ and \eqref{grad-v} and \eqref{vt} hold a.e.~as claimed.
\end{proof}
 
\begin{remark}
It is important to notice that the results of Theorem \ref{thm:lip} hold in the same form when considering a state dependent diffusion coefficient $\sigma(x)$ in \eqref{X}, provided that $\sigma_{ij}\in C^1(\reali^d;\reali)$. Indeed the proof remains exactly the same as we have never used the specific form of the dynamics of $X$ in \eqref{X}.
\end{remark}

There is a simple and useful corollary to the theorem
\begin{corollary}\label{cor:bounds}
Assume $T<+\infty$. Let condition \ref{A$_1$} and one of the two conditions below hold 
\begin{itemize}
\item[(i)] $g(x)=f(T,x)$, $x\in\reali^d$,
\item[(ii)] $\exists\, c>0$ such that $h(T,x)+n(x)\ge -\partial_t f(T,x)-c$, for $x\in\reali^d$.
\end{itemize}
Then for a.e.~$(t,x)\in[0,T]\times\reali^d$ and $\tau_*=\tau_*(t,x)$ we have
\begin{align}\label{vtup1}
\partial_t v(t,x)\le \media\left[\int_0^{\tau_*}\partial_t h(t+s,X^x_s)ds+\partial_t f(t+\tau_*,X^x_{\tau_*})\right]+c\,\prob(\tau_*=T-t)
\end{align}
where $c=0$ if (i) holds.
\end{corollary}
\begin{proof}
Under $(ii)$ the claim is trivial since $\partial_t v\le \overline v$ and recalling \eqref{upvt}. Under $(i)$ instead, we notice that \eqref{grad-5} in the proof of Theorem \ref{thm:lip} may be bounded as follows
\begin{align*}%\label{grad-5}
v(t,x)-v(t-\eps,x)\le&\media\Big[\int_0^\tau\left(h(t+s,X^x_s)-h(t-\eps+s,X^x_s)\right)ds\\
&\hspace{+20pt}+v(t+\tau,X^x_\tau)-v(t-\eps+\tau,X^x_\tau)\Big]\\
\le&\media\Big[\int_0^\tau\left(h(t+s,X^x_s)-h(t-\eps+s,X^x_s)\right)ds\\
&\hspace{+20pt}+f(t+\tau,X^x_\tau)-f(t-\eps+\tau,X^x_\tau)\Big]. 
\end{align*}
Then dividing by $\eps$ and taking limits as $\eps\to0$ we obtain \eqref{vtup1}, thanks to Assumption \ref{ass:fgh}.
\end{proof}

Before concluding the section we provide two simple technical lemmas which will be useful in the next section.
\begin{lemma}\label{lem:gradX}
For $k=1,\ldots d$ one has $\prob$-almost surely
\begin{align}\label{gradX}
\sup_{0\le t\le T}\|\partial_k X^x_t\|_d^2\le 2\exp\left(2T\int^T_0\sum^d_{j=1}\|\nabla_x\mu_j(X^x_s)\|_d^2 ds\right).
\end{align}
\end{lemma}
\begin{proof}
By using $|a+b|^2\le 2(|a|^2+|b|^2)$ and H\"older inequality applied to \eqref{DX} we get
\begin{align*}
\|\partial_k X^x_t\|_d^2=&\sum^d_{j=1}\left(\delta_{j,k}+\int_0^t\langle\nabla_x\mu_j(X^x_s),\partial_kX^x_s\rangle ds\right)^2\\
\le&2\left(1+T\sum_{j=1}^d\int_0^t\|\nabla_x\mu_j(X^x_s)\|_d^2\|\partial_kX^x_s\|_d^2ds\right).
\end{align*}
An application of Gronwall's inequality concludes the proof.
\end{proof} 

\begin{lemma}\label{lem:growth}
Let $q:\reali^d\to \reali$ be Borel-measurable and with growth 
\begin{align}
|q(x)|\le q_0(1+ \|x\|_d^p)
\end{align}
for some $q_0>0$ and $p\ge 1$. Assume $\|\mu(x)\|_d\le C$ for all $x\in\reali^d$ and a given constant $C>0$. Then for $T<+\infty$ and any stopping time $\tau\in[t,T]$ we have  
\begin{align}\label{subl}
\media_{t,x}\left[\mathds{1}_{\{\tau=T\}} q(X_\tau)\right]\le K\,(1+\|x\|_d^p)(T-t)^{-1}\media_{t,x}(\tau-t),\quad (t,x)\in[0,T]\times\reali^d
\end{align} 
where $K>0$ depends only on $q_0$, $d$, $p$, $C$ and $T$.
\end{lemma}
\begin{proof}
Using polynomial growth and H\"older inequality, and by letting $c>0$ be a constant that changes from line to line, we get
\begin{align*}
\media_{t,x}&\left[\mathds{1}_{\{\tau=T\}}q(X_{\tau})\right]\nonumber\\
\le& c\,\media_{t,x}\left[\mathds{1}_{\{\tau=T\}}\left(1+\|X_{\tau}\|_d^p\right)\right]\nonumber\\
\le& c\,\media_{t,x}\left[\mathds{1}_{\{\tau=T\}}\left(1+\|x\|_d^p+\left\|\int_t^{\tau}\mu(X_s)ds\right\|_d^p+\|\sigma (B_{\tau}-B_t)\|_d^p\right)\right]\nonumber\\
\le &c\, \Big(\,(1+\|x\|_d^p)\prob_{t,x}(\tau=T)+\media_{t,x}\Big[\mathds{1}_{\{\tau=T\}}(\tau-t)^p\Big]\nonumber\\
&\hspace{+18pt}+\sqrt{\prob_{t,x}(\tau=T)}\sqrt{\media_{t,x}\left[\|\sigma (B_{\tau}-B_t)\|_d^{2p}\right]}\,\Big)
\nonumber\\
\le &c\, \Big(\,(1+\|x\|_d^p)\prob_{t,x}(\tau=T)+\sqrt{\prob_{t,x}(\tau=T)}\sqrt{\media_{t,x}\Big[(\tau-t)^p\Big]}\,\Big).
\end{align*}
Finally, by using Markov inequality 
\begin{align}\label{eq:markov}
\prob_{t,x}(\tau=T)=\prob_{t,x}(\tau-t\ge T-t)\le (T-t)^{-p}\media_{t,x}\left[(\tau-t)^p\right]
\end{align}
we obtain
\begin{align*}
\media_{t,x}&\left[\mathds{1}_{\{\tau=T\}}q(X_{T})\right]\nonumber\\
\le &c\, \Big(\,(1+\|x\|_d^p)(T-t)^{-1}\media_{t,x}(\tau-t)+\tfrac{1}{(T-t)^{p/2}}\media_{t,x}\Big[(\tau-t)^p\Big]\,\Big)\nonumber\\
\le &c\, \Big(\,1+\|x\|_d^p +(T-t)^{p/2}\Big)(T-t)^{-1}\media_{t,x}(\tau-t).
\end{align*}
It is now immediate to obtain \eqref{subl}.
\end{proof}

\section{Properties of the optimal boundary}\label{sec:lip}

In order to analyse the shape of the continuation set and later on the regularity of its boundary we need to restrict our assumptions. In particular in the following we will often need
\begin{description}
\item [(A$_2$) Terminal value.\namedlabel{A$_2$}{(A$_2$)}] If $T<+\infty$ we have $ \partial_1 g(x)\ge \partial_1 f(T,x)$.
\end{description}

\begin{description}
\item[(B) Drift. \namedlabel{B}{(B)}] For $k=2,\ldots,d$ it holds $\mu_k(x)=\mu_k(x_2,\ldots x_d)$, hence $\partial_1 X^{x,k}\equiv 0$ and from \eqref{DX}
\begin{align}\label{dX1}
\partial_1 X^{x,1}_t=1+\int_0^t\partial_1\mu_1(X^x_s)\partial_1 X^{x,1}_sds.
\end{align}
\item[(C) Spatial monotonicity for $m+h$.  \namedlabel{C}{(C)}] Let $\partial_1 f\in C^{1,2}([0,T]\times\reali^d)$ so that $m(t,x)$ in \eqref{def:m} is continuously differentiable with respect to $x_1$. Assume also that
\begin{align}\label{positive}
\partial_1(h+m)(t,x)\ge 0,\qquad \text{for $(t,x)\in[0,T]\times\reali^d$,}
\end{align}
and that for any compact $K\subseteq [0,T]\times\reali^d$ we have
\begin{align*}
\sup_{(t,x)\in K}\media\left[\int_0^{T-t}|\partial_1 m(t+s,X^x_s)|^2ds\right]<+\infty.
\end{align*}
\end{description}
Condition \ref{B} allows a tractable expression for $\partial_1 X^{x,1}$ (see \eqref{dX1}), which together with condition \ref{C} provide a simple way of determining the shape of the continuation region (see the next proposition). 
\begin{proposition}\label{b:existence}
Assume conditions \ref{A$_1$}, \ref{A$_2$}, \ref{B} and \ref{C}. Then the stopping region is characterised by a free boundary  
\begin{align}
b:[0,T]\times \reali^{d-1}\to\reali\cup\{\pm\infty\}
\end{align}
such that 
\begin{align}
\cS=\{(t,x_1,x_2,\ldots x_d)\in[0,T]\times \reali^{d}: x_1\le b(t,x_2,\ldots x_d)\}
\end{align}
\end{proposition}
\begin{proof}
In order to prove the claim it is sufficient to show that $\partial_1(v-f)(t,x)\ge 0$. The latter indeed implies 
\begin{align*}
(t,x_1,x_2,\ldots x_d)\in\cS\implies (t,x'_1,x_2,\ldots x_d)\in\cS\quad\text{for all $x'_1\le x_1$}. 
\end{align*}

The task is relatively easy thanks to \eqref{grad-v}. Notice that $\partial_1 X^x=(\partial_1 X^{x,1},0,\ldots 0)$ due to condition \ref{B} and 
\begin{align*}
\partial_1 X^{x,1}_t=\exp\left(\int_0^t\partial_1\mu_1(X^x_s)ds\right).
\end{align*}
Then, using that $\partial_1 g(x)\ge \partial_1 f(T,x)$ due to condition \ref{A$_2$}, for a.e.~$(t,x)\in[0,T]\times\reali^d$ we get
\begin{align}
\label{monot1}
\partial_1 v(t,x)=&\media\Big[\int_0^{\tau_*}\partial_1 h(t+s,X^x_s)\partial_1 X^{x,1}_s ds+\mathds{1}_{\{\tau_*<T-t\}}\partial_1 f(t+\tau_*,X^x_{\tau_*})\partial_1 X^{x,1}_{\tau_*}\nonumber\\
&\hspace{+15pt}+\mathds{1}_{\{\tau_*=T-t\}}\partial_1 g(X^x_{T-t})\partial_1 X^{x,1}_{T-t} \Big]\\
\ge& \media\Big[\int_0^{\tau_*}\partial_1 h(t+s,X^x_s)\partial_1 X^{x,1}_s ds+\partial_1 f(t+\tau_*,X^x_{\tau_*})\partial_1 X^{x,1}_{\tau_*}\Big].\nonumber
\end{align}
Now an application of Dynkin's formula gives
\begin{align*}
\media&\left[ \partial_1 X^{x,1}_{\tau_*}\partial_1 f(t+\tau_*,X^x_{\tau_*})\right]\\
=& \partial_1 f(t,x)+\media\Big[\int_0^{\tau_*}e^{\int_0^s\partial_1\mu_1(X^x_u)du}\left(\partial_t \partial_1 f+\cL (\partial_1 f)+\partial_1\mu_1\partial_1 f\right)(t+s,X^x_s)ds\Big]\\
=& \partial_1 f(t,x)+\media\Big[\int_0^{\tau_*}e^{\int_0^s\partial_1\mu_1(X^x_u)du}\partial_1 m(t+s,X^x_s)ds\Big]
\end{align*}
where we have used the easily verifiable equality $\partial_1 m(t,x)=\left(\partial_t \partial_1 f+\cL (\partial_1 f)+\partial_1\mu_1\partial_1 f\right)(t,x)$.

Plugging the last expression in \eqref{monot1} and rearranging terms gives
\begin{align}\label{monot2}
\partial_1(v-f)(t,x)\ge \media\Big[\int_0^{\tau_*}e^{\int_0^s\partial_1\mu_1(X^x_u)du}\,\partial_1( h + m)(t+s,X^x_s)ds\Big]\ge 0
\end{align}
thanks to \eqref{positive}. 
\end{proof}

\begin{remark}
It should be clear that a completely symmetric result holds if \eqref{positive} is replaced by $\partial_1(h+m)\le 0$ and similarly one assumes $\partial_1 g(x)\le \partial_1 f(T,x)$. In such case arguments analogous to the ones employed to prove Proposition \ref{b:existence} may be instead used to prove that $\cS=\{(t,x_1,x_2,\ldots x_d)\in[0,T]\times \reali^{d}: x_1\ge b(t,x_2,\ldots x_d)\}$. 
\end{remark}

If \eqref{positive} holds, then for each $(t,x_2,\ldots x_d)\in[0,T]\times\reali^{d-1}$ we can define 
\begin{align}\label{theta}
\gamma(t,x_2,\ldots x_d):=\inf\{x_1\in\reali:(h+m)(t,x)>0\}
\end{align}
with $\gamma(t,x_2,\ldots x_d)=+\infty$ if the set is empty. It follows from standard arguments that the set 
\begin{align*}
\mathcal{R}:=\{(t,x)\in[0,T]\times\reali^d:x_1>\gamma(t,x_2,\ldots x_d)\}
\end{align*}
is contained in $\cC$. Obviously if $d=1$ then $\gamma$ is a function of time only.

In the next sections we state the results concerning regularity of the optimal boundary.
 
\subsection{Lipschitz boundary for $d=1$}

Here we prove Lipschitz continuity of $b$ in the case $d=1$ and for that we are going to need 
\begin{description}
\item[(D) Bounds I. \namedlabel{D}{(D)}] Let $f\in C^{2,3}([0,T]\times\reali)$ so that $m\in C^{1}([0,T]\times\reali)$. There exists $c>0$ such that $\partial_t(h+m)(t,x)\le c(1+\partial_1(h+m)(t,x))$ for $(t,x)\in[0,T]\times\reali$ and, for any compact $K\subseteq [0,T]\times\reali$, we have
\begin{align*}
\sup_{(t,x)\in K}\media\left[\int_0^{T-t}|\partial_t m(t+s,X^x_s)|^2ds\right]<+\infty.
\end{align*}
\end{description}
We would like to remark that some of the assumptions we make for the proof below may be relaxed when the structure of $f,g,h$ and $\mu$ is known explicitly. This fact will be illustrated in Example 1 in Section \ref{sec:examples}. Notice also that in this setting $\sigma\in\reali^{d'}$ and $\sigma\sigma^\top=\|\sigma\|^2_{d'}$. In what follows $\overline\cI$ denotes the closure of a set $\cI$.

\begin{theorem}\label{thm:lip-d1}
Assume that $d=1$ and $\sigma\sigma^\top>0$. Assume \ref{A$_1$}, \ref{A$_2$}, \ref{C}, \ref{D} and, if $T<+\infty$, either $(i)$ or $(ii)$ from Corollary \ref{cor:bounds}. 
Assume further that $\partial_1 \mu(x)>-\mu$ for some $\mu>0$ and there exists an interval $\cI:=(t_1,t_2)$ with $\overline\cI\subset [0,T)$ and such that
\begin{itemize}
\item[(i)] $\overline{\gamma}:=\sup_{t\in\overline{\cI}}\gamma(t)<+\infty$ (see \eqref{theta});
\item[(ii)] there exists $r\in(\overline{\gamma},+\infty)$ and $\alpha_r>0$ such that $\partial_1(h+m)(t,x)=\partial_x(h+m)(t,x)\ge \alpha_r$ for $(t,x)\in\overline{\cI}\times(-\infty,r)$.
\end{itemize}
Then, for any $\eps>0$ there is $K_\eps>0$ such that $b$ is (bounded) Lipschitz on $[t_1+\eps,t_2-\eps]$ with Lipschitz constant $K_\eps$.
\end{theorem}

\noindent First we need a technical lemma, whose proof will be given after that of the theorem.
\begin{lemma}\label{lem:tech} 
Under the same assumptions as in Theorem \ref{thm:lip-d1} we have 
\begin{itemize}
\item[(a)] for any $t'_1<t'_2$ such that $[t'_1,t'_2]\subseteq\overline \cI$ it holds $\big((t'_1,t'_2)\times\reali\big)\cap\cS\neq\varnothing$;
\item[(b)] $\lim_{x\to-\infty} (v-f)(t,x)=0$ for all $t\in \cI$.
\end{itemize}
\end{lemma}
\begin{proof}[{\bf Proof of Theorem \ref{thm:lip-d1}}]
We provide a full proof only for $T<+\infty$ but the same arguments hold for $T=+\infty$ up to simple minor changes.
We set $w:=v-f$ and we use the notation $\partial_x$ in place of $\partial_1$ and $\mu'(x):=\partial_x\mu(x)$. 

In this setting the diffusion is uniformly non degenerate. Therefore, letting $U$ be an open rectangle in the $(t,x)$-plane whose closure $\overline U$ is contained in $\cC$, and denoting its parabolic boundary by $\partial_P U$, we have that $v\in C^{1,2}(U)\cap C(\overline U)$ is the unique classical solution of the boundary value problem
\begin{align}
\partial_t u+\cL\,u=0,\quad\text{on $U$ and $u|_{\partial_P U}=v|_{\partial_P U}$}.
\end{align} 
For a proof of this standard result one can consult the proof of Theorem 2.7.7 in \cite{KS98}.
In particular, we will use below that $\partial_xv$ and $\partial_t v$ are continuous in $\cC$ (away from the boundary).

The free boundary $b$ is the zero-level set of $w$. Since $\partial_x w\ge 0$ (see \eqref{monot2}), $w$ is continuous and $(b)$ in Lemma \ref{lem:tech} holds, we can find $\delta>0$ sufficiently small so that the equation $w(t,x)=\delta$ has a solution $x=b_\delta(t)>-\infty$ for all $t\in\cI$. Further, by assumption $(ii)$ we have $\partial_x w(t,x)>0$ in $[\cI\times(-\infty,r)]\cap\cC$. Therefore the $\delta$-level set of $w$ is locally given by a function $b_\delta\in C(\cI)$. Moreover $b_\delta(t)>b(t)$ for all $t\in \cI$ (this trivially holds if $b(t)=-\infty$ for some $t\in\cI$). 

On $\cI$ the family $(b_\delta)_{\delta>0}$ decreases as $\delta\to0$ so that its limit $b_0$ exists (possibly equal to $-\infty$). The mapping $t\mapsto b_0(t)$, is upper semi-continuous (on the extended real line), as decreasing limit of continuous functions, and $b_0(t)\ge b(t)$. Since $w(t,b_\delta(t))=\delta$ it is clear that taking limits as $\delta\to0$ we get $w(t,b_0(t))=0$ and therefore $b_0(t)\le b(t)$ so that we conclude
\begin{align}\label{lim-p1}
\lim_{\delta\to0} b_\delta(t)=b(t)\quad \text{for $t\in \cI$}.
\end{align}

Since $v$ is continuously differentiable in $\cC$ and for all $t\in\cI$ it holds $(t,b_\delta(t))\in\cC$ with $\partial_x w(t,b_\delta(t))>0$, then the implicit function theorem gives 
\begin{align}\label{eq:gab1}
b'_\delta(t)=-\frac{\partial_t w(t,b_\delta(t))}{\partial_x w(t,b_\delta(t))},\qquad{t\in\cI}.
\end{align} 

Now we aim at showing that for arbitrary $\eps>0$, letting $\cI_\eps:=(t_1+\eps,t_2-\eps)$ we have 
\begin{eqnarray}
\label{unif-b00}&\text{$b_\delta$ is bounded from below on $\cI_\eps$, uniformly in $\delta$ and}&\\
\label{unif-b0}&\text{there exists $K_\eps>0$ such that $|b'_\delta(t)|\le K_\eps$ on $\cI_\eps$ uniformly in $\delta$}.&
\end{eqnarray}

If \eqref{unif-b00}--\eqref{unif-b0} hold, then by Ascoli-Arzel\`a's theorem we can extract a sequence $(b_{\delta_j})_{j\ge 1}$ such that $b_{\delta_j}\to g$ uniformly on $\cI_\eps$ as $j\to\infty$, where $g$ is Lipschitz continuous with constant $K_\eps$. Uniqueness of the limit implies $g=b$ on $\cI_\eps$ and therefore $b$ is (bounded) Lipschitz on $\cI_\eps$ with constant $K_\eps$.

The rest of the proof will be devoted to verifying \eqref{unif-b00}--\eqref{unif-b0} or equivalently to finding a uniform bound for $b_\delta$ on $\cI_\eps$ and a uniform bound for
\begin{align*}
|b'_\delta(t)|=\frac{|\partial_t w(t,b_\delta(t))|}{\partial_x w(t,b_\delta(t))}\qquad \text{for~$t\in\cI_\eps$.}
\end{align*}
The proof is divided in steps.
\vspace{+5pt}

\emph{Step 1. (Upper bound for $b(t)$).} Due to $(i)$ we have $b(t)<r$ for $t\in\overline{\cI}$ for $r$ as in $(ii)$.
\vspace{+5pt}

\emph{Step 2. (Lower bound for $b(t)$ and $b'(t)$).} 
For any $(t,x)\in[0,T]\times\reali$ and any stopping time $\tau\in[0,T-t]$ we have
\begin{align}\label{ft}
\partial_t f(t,x)=\media\left[\partial_tf(t+\tau,X^{x}_\tau)-\int_0^\tau\partial_tm(t+s,X^{x}_s)ds\right].
\end{align}
Take $(t,x)\in[0,T)\times\reali$ and $\tau_*=\tau_*(t,x)$. Using \eqref{vtup1}, condition \ref{D}, \eqref{monot2} and $\mu'(x)\ge -\mu$ we easily obtain 
\begin{align}\label{eq:gab1.1}
\partial_t w(t,x)\le& \media\Big[\int_0^{\tau_*}\partial_t(h+m)(t+s,X^{x}_s)ds\Big]+c'\,\prob(\tau_*=T-t)\nonumber\\
\le & c\left((1+(T-t)^{-1})\media[\tau_*]+\media\Big[\int_0^{\tau_*}\partial_x(h+m)(t+s,X^{x}_s)ds\Big]\right)\\
\le &c\left((1+(T-t)^{-1})\media[\tau_*]+e^{\mu(T-t)}\partial_x w(t,x)\right)\nonumber
\end{align}
where $c$ and $c'$ are constants and in the second line we have also used Markov inequality as in \eqref{eq:markov}.
Using \eqref{eq:gab1.1} in \eqref{eq:gab1} gives us a lower bound for $b'_\delta$ of the form
\begin{align}\label{lbb1}
b'_\delta(t)\ge -C\left(1 + \frac{\varphi(t,b_\delta(t))}{\partial_x w(t,b_\delta(t))}\right)\qquad\text{for~$t\in\cI$}
\end{align}
where $\varphi(t,x):=\media[\tau_*(t,x)]=\media_{t,x}[\tau_*-t]$. 
Next we want to find a bound for $(\varphi/\partial_x w) (t,b_\delta(t))$.

Recalling \eqref{monot2} we denote $\hat{w}$ the function
\begin{align}\label{w2}
\hat{w}(t,x)=\media_{t,x}\left[\int_t^{\tau_*}e^{\int_t^s\mu'(X_u)du}\partial_x(h+m)(s,X_s)ds\right]
\end{align}
so that $\hat{w}(t,x)\le \partial_x w(t,x)$. It is important to notice that 
\begin{align}\label{inf}
\hat{w}_r:=\inf_{t\in\cI}\hat{w}(t,r)>0
\end{align}
since the segment $[t_1,t_2]\times\{r\}$ is strictly above the optimal boundary and $\partial_x(h+m)>0$ in a neighborhood of the segment (see also Lemma \ref{lem:wxLB}).

Fix $t\in\cI_{\eps/2}=(t_1+\eps/2,t_2-\eps/2)$ and, to simplify notation, set $x_\delta:=b_\delta(t)$ and $\tau_\delta:=\tau_*(t,x_\delta)$. Let 
\begin{align}
\tau_r:=\inf\{s\ge t :(s,X_s)\notin (t_1,t_2)\times(-\infty,r)\}\qquad\text{$\prob_{t,x_\delta}$-a.s.}
\end{align}
and notice that $\tau_r<T$, $\prob_{t,x_\delta}$-a.s.~since $\overline\cI\subset[0,T)$. Tower property of conditional expectation and \eqref{w2} give 
\begin{align*}
\hat{w}&(t,x_\delta)\nonumber\\
=&\media_{t,x_\delta}\Big[\int_t^{\tau_r\wedge\tau_\delta}e^{\int_t^s\mu'(X_u)du}\partial_x(h+m)(s,X_s)ds
+\mathds{1}_{\{\tau_r<\tau_\delta\}}\int_{\tau_r}^{\tau_\delta}e^{\int_t^s\mu'(X_u)du}\partial_x(h+m)(s,X_s)ds\Big]\nonumber\\
=&\media_{t,x_\delta}\Big[\int_t^{\tau_r\wedge\tau_\delta}e^{\int_t^s\mu'(X_u)du}\partial_x(h+m)(s,X_s)ds\nonumber\\
&\hspace{+25pt}+\mathds{1}_{\{\tau_r<\tau_\delta\}}e^{\int_t^{\tau_r}\mu'(X_u)du}\media_{t,x_\delta}\Big(\int_{\tau_r}^{\tau_r+\tau_\delta\circ\theta_{\tau_r}}e^{\int_t^s\mu'(X_u)du}\partial_x(h+m)(s,X_s)ds\Big|\cF_{\tau_r}\Big)\Big],
\end{align*}
where we recall that $\theta_\cdot$ is the shift operator on the canonical space.
Strong Markov property gives
\begin{align}\label{w3}
\hat{w}&(t,x_\delta)\nonumber\\
=&\media_{t,x_\delta}\Big[\int_t^{\tau_r\wedge\tau_\delta}e^{\int_t^s\mu'(X_u)du}\partial_x(h+m)(s,X_s)ds\nonumber\\
&\hspace{+25pt}+\mathds{1}_{\{\tau_r<\tau_\delta\}}e^{\int_t^{\tau_r}\mu'(X_u)du}\media_{\tau_r,X_{\tau_r}}\Big(
\int_{\tau_r}^{\tau_\delta}e^{\int^s_{\tau_r}\mu'(X_u)du}\partial_x(h+m)(s,X_s)ds\Big)\Big]\\
=&\media_{t,x_\delta}\Big[\int_t^{\tau_r\wedge\tau_\delta}e^{\int_t^{s}\mu'(X_u)du}\partial_x(h+m)(s,X_s)ds+\mathds{1}_{\{\tau_r<\tau_\delta\}}e^{\int_t^{\tau_r}\mu'(X_u)du}\hat{w}(\tau_r,X_{\tau_r})\Big].\nonumber
\end{align}
For $\hat{w}$ we therefore have the following lower bound 
\begin{align}\label{lw1}
\hat{w}(t,x_\delta)\ge C'\left(\media_{t,x_\delta}\left[\alpha_r(\tau_\delta\wedge\tau_r-t)+ \mathds{1}_{\{\tau_r<\tau_\delta\}}\hat{w}(\tau_r,X_{\tau_r})\right]\right)
\end{align}
with $C'=e^{-\mu T}$.
The same argument applied to $\varphi$ gives
\begin{align}\label{phi}
\varphi(t,x_\delta)=\media_{t,x_\delta}\left[(\tau_\delta\wedge\tau_r-t)+\mathds{1}_{\{\tau_r<\tau_\delta\}}\varphi(\tau_r,X_{\tau_r})\right].
\end{align}

Now from \eqref{phi} and \eqref{lw1}, and by noticing that $0\le \varphi(t,x)\le T$ we obtain
\begin{align}\label{lip2}
0\le& \frac{\varphi(t,x_\delta)}{\partial_x w(t,x_\delta)}\le\frac{\varphi(t,x_\delta)}{\hat{w}(t,x_\delta)}\nonumber\\
\le& \frac{1}{C'\alpha_r}+\frac{\media_{t,x_\delta}\left[\mathds{1}_{\{\tau_r<\tau_\delta\}}\varphi(\tau_r,X_{\tau_r})\right]}{C'\left(\media_{t,x_\delta}\left[\alpha_r(\tau_\delta\wedge\tau_r-t)+ \mathds{1}_{\{\tau_r<\tau_\delta\}}\hat{w}(\tau_r,X_{\tau_r})\right]\right)}\nonumber\\
\le& C_r\left(1+\frac{\prob_{t,x_\delta}(\tau_r<\tau_\delta,\tau_r<t_2)+\prob_{t,x_\delta}(\tau_r<\tau_\delta,\tau_r=t_2)}{\media_{t,x_\delta}\left[\alpha_r(\tau_\delta\wedge\tau_r-t)+ \mathds{1}_{\{\tau_r<\tau_\delta\}}\hat{w}(\tau_r,X_{\tau_r})\right]}\right),
\end{align}
where $C_{r}:=(\alpha_r^{-1}\vee T)/C'$. Since $t\in\cI_{\eps/2}=(t_1+\eps/2,t_2-\eps/2)$ we have
\begin{align}\label{lip3}
&\frac{\prob_{t,x_\delta}(\tau_r<\tau_\delta,\tau_r<t_2)}{\media_{t,x_\delta}\left[\alpha_r(\tau_\delta\wedge\tau_r-t)+ \mathds{1}_{\{\tau_r<\tau_\delta\}}\hat{w}(\tau_r,X_{\tau_r})\right]}\nonumber\\
&\le\frac{\prob_{t,x_\delta}(\tau_r<\tau_\delta,\tau_r<t_2)}{\media_{t,x_\delta}\left[\mathds{1}_{\{\tau_r<\tau_\delta\}\cap\{\tau_r<t_2\}}\hat{w}(\tau_r,r)\right]}\le\frac{\prob_{t,x_\delta}(\tau_r<\tau_\delta,\tau_r<t_2)}{\hat{w}_r\prob_{t,x_\delta}\left(\tau_r<\tau_\delta,\tau_r<t_2\right)}\le\frac{1}{\hat{w}_r}
\end{align}
by using that $\mathds{1}_{\{\tau_r<t_2\}}X_{\tau_r}=r$ and recalling \eqref{inf}. Similarly, for $t\in\cI_{\eps/2}$, we have
\begin{align}\label{lip4}
&\frac{\prob_{t,x_\delta}(\tau_r<\tau_\delta,\tau_r=t_2)}{\media_{t,x_\delta}\left[\alpha_r(\tau_\delta\wedge\tau_r-t)+ \mathds{1}_{\{\tau_r<\tau_\delta\}}\hat{w}(\tau_r,X_{\tau_r})\right]}\nonumber\\
&\le\frac{\prob_{t,x_\delta}(\tau_r<\tau_\delta,\tau_r=t_2)}{\media_{t,x_\delta}\left[\mathds{1}_{\{\tau_r<\tau_\delta\}\cap\{\tau_r=t_2\}}\alpha_r(\tau_\delta\wedge\tau_r-t)\right]}\le\frac{1}{\alpha_r(t_2-t)}\le \frac{2}{\alpha_r\eps}.
\end{align}
Now plugging the last two estimates back into \eqref{lip2} and recalling \eqref{lbb1} we finally conclude  
\begin{align}\label{lo-bp}
b'_\delta(t)\ge - C\left(1+C_r\left(1+\hat{w}_r^{-1} +2(\alpha_r\eps)^{-1}\right)\right)=:-\beta_{\eps,r}\quad\text{for all $t\in\overline{\cI_{\eps/2}}$.}
\end{align}

Thanks to $(a)$ in Lemma \ref{lem:tech} we can find $t_0\in\cI$, arbitrarily close to $t_1$ and such that $|b(t_0)|<\infty$. So with no loss of generality we assume $t_0=t_1+\eps/2$ and $b(t_0)>-\infty$. Since the bound in \eqref{lo-bp} is uniform in $\delta$, then \eqref{lim-p1} implies that $b(t)\ge b(t_0)-\beta_{\eps,r}|t-t_0|$ for all $t\in\overline{\cI_{\eps/2}}$ and proves \eqref{unif-b00}. Hence, there exits $r^\eps<r$ such that $b(t)\in(r^\eps,r)$ for all $t\in\overline{\cI_{\eps/2}}$. This fact will be used in the next step of the proof. 
\vspace{+5pt}

\emph{Step 3. (Upper bound for $b'(t)$).} It remains to find an upper bound for $b'_\delta$ on $\overline{\cI_\eps}$ which is uniform in $\delta$. For that it is convenient to denote
\begin{align}\label{gtilde}
\widetilde{g}(x):=|h(T,x)+n(x)|+2|\partial_tf(T,x)|\quad\text{for $x\in\reali$.}
\end{align} 
and recall that $\partial_t v(t,x)\ge \underline{v}(t,x)$ given in \eqref{downvt}. Using again \eqref{ft} we immediately find
\begin{align}\label{ov-w}
\partial_t w(t,x)\ge \overline{w}(t,x):=\media\left[\int_0^{\tau_*}\partial_t(h+m)(t+s,X^x_s)ds-\mathds{1}_{\{\tau_*=T-t\}}\widetilde{g}(X^x_{T-t})\right]
\end{align}
for~$(t,x)\in[0,T]\times\reali$.
Now, for $t\in\cI_\eps$ we set 
\begin{align}
\tau'_r:=\inf\{s\ge t:(s,X_s)\notin (t_1+\eps,t_2-\eps/2)\times(-\infty,r)\}\qquad\text{$\prob_{t,x_\delta}$-a.s.,}
\end{align}
and, using the strong Markov property as in step 2 above, we find
\begin{align}\label{4.28}
\overline{w}(t,x_\delta)=\media_{t,x_\delta}\left[\int_t^{\tau'_r\wedge\tau_\delta}\partial_t(h+m)(s,X_s)ds+\mathds{1}_{\{\tau'_r<\tau_\delta\}}\overline{w}(\tau'_r,X_{\tau'_r})\right].
\end{align}

Since $b(t)\in[r^\eps,r]$ for $t\in\overline{\cI_{\eps/2}}$ then under $\prob_{t,x_\delta}$ we have $X_{s}\in[r^\eps,r]$ for all $s\in(t,\tau'_r\wedge\tau_\delta)$ and therefore there exists $\nu_\eps>0$ such that $\partial_t(h+m)(s,X_{s})\ge -\nu_\eps$ for all $s\in(t,\tau'_r\wedge\tau_\delta)$. On the other hand, from the definition of $\overline w$ in \eqref{ov-w} and properties of $\partial_t(h+m)$ and $\widetilde{g}$ (see Assumption \ref{ass:fgh}) it also follows
\begin{align}\label{sup}
\overline{w}_{r,\eps}:=\sup_{t\in\cI}|\overline{w}(t,r)|+\sup_{x\in[r^\eps,r]}|\overline{w}(t_2-\eps/2,x)|<\infty.
\end{align}
In conclusion, from \eqref{4.28} we have
\begin{align}
\overline{w}(t,x_\delta)\ge -\nu_\eps\media_{t,x_\delta}\left[(\tau_\delta\wedge\tau'_r-t)\right]-\overline{w}_{r,\eps}\prob_{t,x_\delta}(\tau'_r<\tau_\delta).
\end{align}

Using \eqref{w3} with $\tau_r$ replaced by $\tau'_r$ we obtain \eqref{lw1} with $\tau_r$ replaced by $\tau'_r$. Hence, recalling \eqref{eq:gab1}, for~$t\in\cI_\eps$ we have
\begin{align}
b'_\delta(t)\le &\frac{\nu_\eps\media_{t,x_\delta}\left[(\tau_\delta\wedge\tau'_r-t)\right]+\overline{w}_{r,\eps}\prob_{t,x_\delta}(\tau'_r<\tau_\delta)}{C'\left(\media_{t,x_\delta}\left[\alpha_r(\tau_\delta\wedge\tau'_r-t)+ \mathds{1}_{\{\tau'_r<\tau_\delta\}}\hat{w}(\tau'_r,X_{\tau'_r})\right]\right)}\nonumber\\
\le &\frac{1}{C'}\left(\frac{\nu_\eps}{\alpha_r}+\frac{\overline{w}_{r,\eps}\prob_{t,x_\delta}(\tau'_r<\tau_\delta)}{\media_{t,x_\delta}\left[\alpha_r(\tau_\delta\wedge\tau'_r-t)+ \mathds{1}_{\{\tau'_r<\tau_\delta\}}\hat{w}(\tau'_r,X_{\tau'_r})\right]}\right).
\end{align}
The last term on the right-hand side above may be estimated by using the same arguments as in \eqref{lip3} and \eqref{lip4}, upon noticing that the argument in \eqref{lip4} carries over to this case since for $t\in\cI_\eps$ we have $\mathds{1}_{\{\tau'_r<\tau_\delta\}\cap\{\tau'_r=t_2-\eps/2\}}(\tau_\delta\wedge\tau'_r-t)\ge \eps/2$. 

Therefore we conclude 
\begin{align}\label{up-bp}
b'_\delta(t)\le C'_r\left(1+2(\alpha_r\eps)^{-1}+\hat{w}_r^{-1}\right) 
\end{align}
where $C'_r=[(\nu_\eps/\alpha_r)\vee\, \overline{w}_{r,\eps}]/C'$. Now \eqref{up-bp} and \eqref{lo-bp} imply \eqref{unif-b0} and the proof is complete.
\end{proof}
\begin{proof}[{\bf Proof of Lemma \ref{lem:tech}}]
We set $w:=v-f$ and use the notation $\partial_x$ in place of $\partial_1$. By a simple application of Dynkin's formula we can write $w$ as
\begin{align}\label{w}
w(t,x)=\media\left[\int_0^{\tau_*}(h+m)(t+s,X^x_s)ds+\mathds{1}_{\{\tau_*=T-t\}}\big(g(X_{\tau_*})-f(t+\tau_*,X^x_{\tau_*})\big)\right].
\end{align}
For future reference we notice that under $(ii)$ of Theorem \ref{thm:lip-d1} we have $\gamma\in C^1(\overline{\cI})$ by the implicit function theorem, since $(h+m)(t,\gamma(t))=0$. Therefore 
\[
\underline\gamma:=\inf_{t\in\overline\cI}\gamma(t)>-\infty.
\]
\vspace{+4pt}

{\em Step 1.} Here we prove $(a)$ by contradiction. Assume that we can find $t'_1<t'_2$ in $\overline \cI$ such that $\big((t'_1,t'_2)\times\reali\big)\cap\cS=\varnothing$. Fix $t\in(t'_1,t'_2)$, take $x\le \underline\gamma$ and define
\[
\rho_\gamma(t,x):=\inf\{s\ge 0\,:\,X^x_s\ge \underline \gamma \}\wedge(t'_2-t).
\]
By assumption we have $\tau_*(t,x)\ge (t'_2-t)$, hence $\tau_*(t,x)\ge \rho_\gamma(t,x)$, $\prob$-a.s. Recall that $(h+m)(s,y)\le 0$ for $y\le \gamma(s)$, $s\in[0,T]$ and notice that 
\begin{align}\label{boundhm}
(h+m)(s,y)=&(h+m)(s,\underline \gamma)-\int_y^{\underline \gamma}\partial_x(h+m)(s,z)dz\le -\alpha_r(\underline \gamma-y),
\end{align}
thanks to $(ii)$ in Theorem \ref{thm:lip-d1}, for $s\in\cI$ and $y\le \underline \gamma$. Denote $\rho_\gamma=\rho_\gamma(t,x)$. Similarly to, e.g., \eqref{w3}, we can use the tower property of conditional expectation and strong Markov property in \eqref{w}, along with \eqref{boundhm}. This gives 
\begin{align}\label{wtech}
w(t,x)=&\media\left[\int_0^{\rho_\gamma}(h+m)(t+s,X^x_s)ds+w(t+\rho_\gamma,X^x_{\rho_\gamma})\right]\notag\\
\le&-\alpha_r\media\left[\int_0^{\rho_\gamma}\big(\underline \gamma- X^x_s\big)ds\right]+ w_\gamma,
\end{align}
where $w_\gamma:=\sup_{t\le s\le t'_2}w(s,\underline \gamma)<+\infty$, since $w$ is continuous, and we used that $w(t+\rho_\gamma,X^x_{\rho_\gamma})\le w_\gamma$ since $\partial_x w\ge 0$ by \eqref{monot2}. Finally, letting $x\downarrow-\infty$ we reach a contradiction because $\rho_\gamma(t,x)\uparrow(t'_2-t)$ and the first term in the last expression of \eqref{wtech} goes to $-\infty$.
\vspace{+4pt}

{\em Step 2}. Here we prove $(b)$. Pick any $t\in\cI$. From the previous step we know that we can find $t'\in\cI$ with $t'>t$ such that $|b(t')|<+\infty$. In particular, this implies that $\{t'\}\times (-\infty,b(t')]\subseteq\cS$. Now, for $x\le\underline\gamma\wedge b(t')$ let 
\[
\rho'_\gamma(t,x):=\inf\{s\ge 0\,:\, X^x_s\ge \underline \gamma \wedge b(t')\}
\] 
and notice that $\rho'_\gamma(t,x)\wedge\tau_*(t,x)$ is smaller than the first time the process $(t+s,X^x_s)$ leaves the the set $[t,t')\times (-\infty,\underline \gamma\wedge b(t'))$. Hence, $\rho'_\gamma(t,x)\wedge\tau_*(t,x)\le (t'-t)$, $\prob$-a.s.

Using the tower property of conditional expectation and strong Markov property in \eqref{w}, we obtain
\begin{align*}%\label{wtech}
w(t,x)=&\media\left[\int_0^{\rho'_\gamma\wedge\tau_*}(h+m)(t+s,X^x_s)ds+w(t+\rho'_\gamma\wedge\tau_*,X^x_{\rho'_\gamma\wedge\tau_*})\right]\notag\\
\le&\media\left[\mathds{1}_{\{\rho'_\gamma<\tau_*\}}w(t+\rho'_\gamma,X^x_{\rho'_\gamma})\right],
\end{align*}
where we have also used that $(h+m)(s,y)\le 0$ for $y\le \gamma(s)$, and $\mathds{1}_{\{\tau_*\le \rho'_\gamma\}}w(t+\tau_*,X^x_{\tau_*})=0$, $\prob$-a.s. Next we notice that 
\[
\{\rho'_\gamma<\tau_*\}=\{\rho'_\gamma<\tau_*,\rho'_\gamma<t'-t\}
\] 
because $\{\rho'_\gamma<\tau_*,\rho'_\gamma\ge t'-t\}=\varnothing$. Moreover, we recall that $w(t+\rho'_\gamma,X^x_{\rho'_\gamma})\le \sup_{t\le s\le t'}w(s,\underline \gamma)=:w'_\gamma<+\infty$ as $\partial_x w\ge 0$, by \eqref{monot2}. Then we obtain 
\begin{align*}%\label{wtech}
w(t,x)\le&w'_\gamma \prob\big(\rho'_\gamma(t,x)<t'-t\big)=w'_\gamma \prob\big(\sup_{0\le s\le t'-t} X^x_s>\underline\gamma\wedge b(t')\big).
\end{align*}
Letting $x\to-\infty$ gives us $(b)$, as claimed.
\end{proof}

\begin{remark}\label{rem:C1}
One can use local Lipschitz continuity of $t\mapsto b(t)$ and the law of iterated logarithm to show that $\tau_*$ is actually equal to the first time $X$ goes strictly below the boundary, i.e., for all $(t,x)\in[0,T)\times\reali$, it holds $\prob$-a.s.
\begin{align*}
\inf\{s\ge0\,:\,X^x_s\le b(t+s)\}\wedge (T-t)=\inf\{s\ge0\,:\,X^x_s< b(t+s)\}\wedge (T-t).
\end{align*}
This is an important fact that can be used to prove that $(t,x)\mapsto\tau_*(t,x)$ is continuous $\prob$-a.s., and it is zero at all boundary points, hence implying $v\in C^{1,1}([0,T)\times\reali)$~(see for example \cite[Sec.~5 and 6]{DeAE16}).
\end{remark}

\subsection{Lipschitz boundary for $d\ge 2$}

Lipschitz regularity for optimal boundaries when $d\ge 2$ requires slightly stronger assumptions on the functions $f,g,h$ and $\mu$ which, however, are in line with those originally used in \cite{SS91}. We give the result under two different sets of assumptions, namely conditions \ref{F} and \ref{G} below. The main difference between the two is that in \ref{F} we do not need a positive lower bound for the quantity $\partial_1(h+m)$ but we need to compensate by imposing stronger bounds on the remaining quantities. Under conditions \ref{G} we have a uniform lower bound on $\partial_1(h+m)$ so that other inequalities may be relaxed. 

One should compare \ref{G} to the assumptions in \cite{SS91} and notice that we are in a similar setting (see also Remark \ref{rem:comp} below for further details). We stress here that our theorems below \emph{do not require uniform ellipticity} of the operator $\sigma\sigma^\top$ and therefore could not be obtained by PDE methods employed in \cite{SS91}. We illustrate an application of our results and methodology in Example 2 of Section \ref{sec:examples}, which addresses the case of a degenerate diffusion.

The main idea of the proofs below is again to use the implicit function theorem as in the case of $d=1$ (Theorem \ref{thm:lip-d1}). However here we cannot rely upon continuity of $\nabla_x v$ and $\partial_t v$ due to the lack of uniform ellipticity of the parabolic operator $\partial_t+\cL$ and therefore the arguments from Theorem \ref{thm:lip-d1} do not carry over.
To overcome this additional difficulty we provide some notation and some technical lemmas, and throughout the section we use 
\begin{description}
\item[(E) Regularity $\partial_t m$ and $\partial_k m$. \namedlabel{E}{(E)}] Let $f\in C^{2,3}([0,T]\times\reali^d)$ so that $m\in C^1([0,T]\times\reali^d)$. For any compact $K\subseteq [0,T]\times \reali^d$ we have 
\begin{align*}
\sup_{(t,x)\in K}\media\left[\int_0^{T-t}|\partial_k m(t+s,X^x_s)|^2+|\partial_t m(t+s,X^x_s)|^2ds\right]<+\infty.
\end{align*}
\end{description}

From now on we denote $w:=v-f$. An application of Dynkin formula to $f(t,X_t)$ gives
\begin{align}
w(t,x)=\sup_{0\le \tau\le T-t}\media\left[\int_0^\tau (h+m)(t+s,X^x_s)ds+\mathds{1}_{\{\tau=T-t\}}\left(g(X^x_{T-t})-f(T,X^x_{T-t})\right)\right].
\end{align}
For $k=1,\ldots d$ we introduce the functions 
\begin{align}\label{eq:wcirc-k}
w^\circ_k(t,x):=&\media\Big[\int_0^{\tau_*}\langle\nabla_x(h+m)(t+s,X^x_s),\partial_kX^x_s\rangle ds\nonumber\\
&\hspace{+25pt}+\mathds{1}_{\{\tau_*=T-t\}}\langle\nabla_x\left(g(X^x_{T-t})-f(T,X^x_{T-t})\right),\partial_kX^x_{T-t}\rangle\Big]
\end{align}
and in particular under condition \ref{B} we notice that 
\begin{align}\label{eq:wcirc-1}
w^\circ_1(t,x)=&\media\Big[\int_0^{\tau_*}\partial_1(h+m)(t+s,X^x_s)\partial_1X^{x,1}_s ds\nonumber\\
&\hspace{+25pt}+\mathds{1}_{\{\tau_*=T-t\}}\partial_1\left(g(X^x_{T-t})-f(T,X^x_{T-t})\right)\partial_1X^{x,1}_{T-t}\Big].
\end{align}
Thanks to Theorem \ref{thm:lip} we have $\partial_1 w = w^\circ_1$ almost everywhere and therefore under condition \ref{A$_1$}, for a.e.~$(t,x)\in[0,T]\times\reali^d$, we also have
\begin{align}\label{monot2b}
\partial_1w(t,x)\ge \media\Big[\int_0^{\tau_*}\partial_1(h+m)(t+s,X^x_s)\partial_1X^{x,1}_s ds\Big].
\end{align}
In analogy with \eqref{upvt} and \eqref{downvt} we also introduce
\begin{align}
\label{eq:wtup}\overline w(t,x):=&\media\left[\int_0^{\tau_*}\partial_t(h+m)(t+s,X^x_s)ds-\mathds{1}_{\{\tau_*=T-t\}}\left((h+\partial_tf)(T,X^x_{T-t})+n(X^x_{T-t})\right)\right]\\[+4pt]
\label{eq:wtdwn}\underline w(t,x):=&\media\left[\int_0^{\tau_*}\partial_t(h+m)(t+s,X^x_s)ds-\mathds{1}_{\{\tau_*=T-t\}}\left|(h+\partial_tf)(T,X^x_{T-t})+n(X^x_{T-t})\right|\right].
\end{align}

The next technical lemma will be used in the proofs of Proposition \ref{prop:bdelta} below.

\begin{lemma}\label{lem:wxLB}
Let $\cO$ be a bounded open set in $[0,T]\times\reali^d$ and $K\subset \cO\cap\cC$ a compact. Then if \ref{A$_2$} and \ref{B} hold, and $\partial_1(h+m)$ is continuous and strictly positive on $\cO$, we have
\begin{align}
\inf_{(t,x)\in K}w^\circ_1(t,x)>0
\end{align}
\end{lemma}
\begin{proof}
Letting $\tau_{0}:=\inf\{s\ge t\,:\,(s,X_s)\notin \cO\cap\cC\}$, $\prob_{t,x}$-a.s., we easily obtain from \eqref{eq:wcirc-1}, \ref{A$_2$} and \ref{B} that 
\begin{align}\label{eq:t1}
w^\circ_1(t,x)\ge c_\cO\media_{t,x}\left[\tau_0-t\right]
\end{align}
for some $c_\cO>0$ only depending on $\cO$. Choose now an arbitrary function $\psi\in C^2([0,T]\times\reali^d)$ such that $\psi(\cdot)\ge 1$ on the complement of $\cO\cap\cC$ (denoted $(\cO\cap\cC)^c$) and $\psi(\cdot)\le \tfrac{1}{2}$ on $K$. With no loss of generality $\cL \psi(\cdot)\le c'_\cO$ on $\cO$ for some $c'_\cO>0$ since $\psi\in C^2(\cO)$ and $\cO$ is bounded. Then by an application of Dynkin formula, for any $(t,x)\in K$ we obtain
\begin{align}\label{eq:t2}
\media_{t,x}\left[\tau_0-t\right]=&\,\media_{t,x}\left[\int_t^{\tau_0}dt\right]\ge \frac{1}{c'_\cO}\media\left[\int_t^{\tau_0}\cL\psi (s,X_s)dt\right]\nonumber\\
=&\,\frac{1}{c'_\cO}\left(\media_{t,x}\left[\psi(\tau_0,X_{\tau_0})\right]-\psi(t,x)\right)\ge \frac{1}{2c'_\cO}
\end{align}
where the final inequality uses that $\prob_{t,x}\left[(\tau_0,X_{\tau_0})\in(\cO\cap\cC)^c\right]=1$ so that $\psi(\tau_0,X_{\tau_0})\ge 1$, while $\psi(t,x)\le \tfrac{1}{2}$ on $K$. Plugging \eqref{eq:t2} in \eqref{eq:t1} concludes the proof.
\end{proof}

From now on we let $$c:[0,T]\times\reali^{d-1}\to\reali$$ be a given arbitrary function.
The next lemma is an application of the chain rule and its proof is omitted. 
\begin{lemma}\label{lem:DXb}
Assume that $c$ is differentiable at $(t,x_2,\ldots x_d)$ and denote $$c:=(c(t,x_2,\ldots x_d),x_2,\ldots x_d).$$ 
Then under condition \ref{B} we have $\prob$-a.s.
\begin{align}
\label{DXb0}\partial_t(X^{c,1}_s)&=\partial_t c(t,x_2,\ldots x_d)\partial_1X^{x,1}_s|_{x=c}\\[+4pt]
\label{DXb1}\partial_k (X^{c,1}_s)&=\partial_k c(t,x_2,\ldots x_d)\partial_1 X^{x,1}_s|_{x=c}+\partial_k X^{x,1}_s|_{x=c},\\[+4pt]
\label{DXb2}\partial_k (X^{c,j}_s)&=\partial_k X^{x,j}_s|_{x=c},\quad\text{and}\quad \partial_t(X^{c,j}_s)=0
\end{align}
for all $s\in[0,T]$ and all $k,j\ge 2$. 
\end{lemma} 
Notice that \eqref{DXb2} holds because the $j$-th component of $X^x$, with $j\ge 2$, is not affected by changes in the initial point $x_1=c(t,x_2,\ldots x_d)$ of the first component $X^{x,1}$, due to condition \ref{B}. With a slight abuse of notation we are often going to use $\partial_k X^{c,j}=\partial_k X^{x,j}|_{x=c}$ when no confusion shall arise. For future reference we also give a straightforward corollary.
\begin{corollary}\label{cor:DXb}
Let $G\in C^{1,1}([0,T]\times\reali^d)$ and $F\in C^1(\reali^d)$, then under the assumptions of Lemma \ref{lem:DXb} and with the same notation we have $\prob$-a.s., for all $s\in[0,T-t]$,
\begin{align*}
\frac{\partial}{\partial t}G(t+s,X^{c}_s)=&\partial_t G(t+s,X^{c}_s)+\partial_1 G(t+s,X^{c}_s)\partial_1X^{x,1}|_{x=c}\partial_t c(t,x_2,\ldots x_d),\\
\frac{\partial}{\partial x_k} F(X^{c}_s)=&\langle\nabla_x F(X^{c}_s),\partial_k X^{x}_s|_{x=c} \rangle +\partial_1 F(X^{c}_s)\partial_1 X^{x,1}_s|_{x=c} \partial_k c(t,x_2,\ldots x_d).
\end{align*}
\end{corollary} 
The next lemma provides useful bounds which will be then employed to prove the main theorems below.
\begin{lemma}\label{lem:directional}
Let conditions \ref{A$_1$}, \ref{A$_2$}, \ref{B} and \ref{E} hold. Assume that $c$ is differentiable at $(t,x_2,\ldots x_d)$ and denote
\begin{align*}
&c:=(c(t,x_2,\ldots x_d),x_2,\ldots x_d),\\ 
&c^+_{t,\eps}:=(c(t+\eps,x_2,\ldots x_d),x_2,\ldots x_d),\\
&c^-_{t,\eps}:=(c(t-\eps,x_2,\ldots x_d),x_2,\ldots x_d),\\
&c^+_{k,\eps}:=(c(t,x_2,\ldots x_k+\eps,\ldots x_d),x_2,\ldots x_k+\eps,\ldots x_d),\\
&c^-_{k,\eps}:=(c(t,x_2,\ldots x_k-\eps,\ldots x_d),x_2,\ldots x_k-\eps,\ldots x_d),
\end{align*}
for $k=2,\ldots d$. Then for any $k\ge 2$ we have
\begin{align}\label{Dx-wb}
\limsup_{\eps\to 0}&\frac{w(t,c)-w(t,c^-_{k,\eps})}{\eps}\\
\le& w^\circ_1(t,c) \partial_k c(t,x_2,\ldots x_d)+w^\circ_k(t,c)\le \liminf_{\eps\to 0}\frac{w(t,c^+_{k,\eps})-w(t,c)}{\eps}.\nonumber
\end{align}
Moreover we also have
\begin{align}
\label{Dt-wb1}\limsup_{\eps\to 0}\frac{w(t,c)-w(t-\eps,c^-_{t,\eps})}{\eps}\le& w^\circ_1(t,c) \partial_t c(t,x_2,\ldots x_d)+\overline w(t,c),\\[+4pt]
\label{Dt-wb2}\liminf_{\eps\to 0}\frac{w(t+\eps,c^+_{t,\eps})-w(t,c)}{\eps}\ge& w^\circ_1(t,c) \partial_t c(t,x_2,\ldots x_d)+\underline w(t,c).
\end{align}
\end{lemma}
\begin{proof}
The proof relies on Lemma \ref{lem:DXb} and on arguments similar to those used to prove Theorem \ref{thm:lip}.
Denote $\tau=\tau_*(t,c)$ and let us consider the first inequality in \eqref{Dx-wb}. Notice that $\tau$ is optimal for $w(t,c)$ and sub-optimal for $w(t,c^-_{k,\eps})$ for $\eps>0$. Therefore we may estimate
\begin{align*}
w(t,c)-w(t,c^-_{k,\eps})\le& \media\left[\int_0^\tau\left((h+m)(t+s,X^{c}_s)-(h+m)(t+s,X^{c^-_{k,\eps}}_s)\right)ds\right]\\
&+\media\left[\mathds{1}_{\{\tau=T-t\}}\left(g(X^c_{T-t})-g(X^{c^-_{k,\eps}}_{T-t})+f(T,X^{c^-_{k,\eps}}_{T-t})-f(T,X^c_{T-t})\right)\right].
\end{align*}
Dividing by $\eps$, taking limits as $\eps\to0$ and using Corollary \ref{cor:DXb} and \eqref{grad-L2} we obtain the first inequality in \eqref{Dx-wb} upon recalling the definitions of $w^\circ_1$ and $w^\circ_k$.

A symmetric argument may be applied to obtain the second inequality in \eqref{Dx-wb}. This time we notice that $\tau$ is sub-optimal for $w(t,c^+_{k,\eps})$ for $\eps>0$ so that 
\begin{align*}
w(t,c^+_{k,\eps})-w(t,c)\ge& \media\left[\int_0^\tau\left((h+m)(t+s,X^{c^+_{k,\eps}}_s)-(h+m)(t+s,X^{c}_s)\right)ds \right]\\
&+\media\left[\mathds{1}_{\{\tau=T-t\}}\left(g(X^{c^+_{k,\eps}}_{T-t})-g(X^{c}_{T-t})+f(T,X^{c}_{T-t})-f(T,X^{c^+_{k,\eps}}_{T-t})\right)\right]
\end{align*}
holds. Dividing by $\eps$ and passing to the limit the claim follows thanks to Corollary \ref{cor:DXb}.

For \eqref{Dt-wb1} we repeat arguments similar to those that led to \eqref{eq:new1} in step 2 of the proof of Theorem \ref{thm:lip}. These give
\begin{align*}
w(t,c)&-w(t-\eps,c^-_{t,\eps})\\
\le&\media\left[\int_0^\tau\left((h+m)(t+s,X^c_s)-(h+m)(t-\eps+s,X^{c^{-}_{t,\eps}}_s)\right)ds\right]
\end{align*}
\begin{align*}%\\[+4pt]
&+\media\left[\mathds{1}_{\{\tau=T-t\}}\left(\hat g(X^{c}_{T-t})-\hat g(X^{c^{-}_{t,\eps}}_{T-t})\right)\right]\\[+4pt]
&-\media_{c^{-}_{t,\eps}}\left[\mathds{1}_{\{\tau=T-t\}}\media_{X_{T-t}}\left(\int_0^\eps [(h+\partial_t f)(T-\eps+s,X_s)+n(X_s)]ds\right)\right],
\end{align*}
where we have set $\hat{g}(x):=g(x)-f(T,x)$ to simplify the notation. 
Now dividing both sides of the above expression by $\eps$, letting $\eps\to0$ and using Corollary \ref{cor:DXb} we get \eqref{Dt-wb1}.

Similar arguments hold for \eqref{Dt-wb2} and following step 2 in the proof of Theorem \ref{thm:lip}, using that $\hat g\ge0$, we have
\begin{align*}
w(t+\eps,c^+_{t,\eps})&-w(t,c)\\[+4pt]
\ge &\media\left[\int_0^{\tau\wedge(T-t-\eps)}\left((h+m)(t+\eps+s,X^{c^+_{t,\eps}}_s)-(h+m)(t+s,X^{c}_s)\right)ds\right]\\[+4pt]
&+\media\left[\mathds{1}_{\{\tau>T-t-\eps\}}\left(\hat g(X^{c^{+}_{t,\eps}}_{T-t-\eps})-\hat g(X^{c}_{T-t-\eps})\right)\right]\\[+4pt]
&-\media_c\left[\mathds{1}_{\{\tau>T-t-\eps\}}\media_{X_{T-t-\eps}}\left(\int_0^\eps|(h+\partial_tf)(T-\eps+s,X_s)+n(X_s)|ds\right)\right].
\end{align*}
Finally, dividing by $\eps$ and passing to the limit we get \eqref{Dt-wb2} thanks to Corollary \ref{cor:DXb}.
\end{proof}

Under the assumptions of Proposition \ref{b:existence} we have $\partial_1w\ge 0$ almost everywhere and therefore, for each $\delta>0$, and for fixed $( t, x_2,\ldots  x_d)$, the equation $w( t, \cdot,x_2,\ldots x_d)=\delta$ has at most a unique solution which we denote by $b_\delta(t,x_2,\ldots x_d)$. 

The next proposition states that $b_\delta$ is Lipschitz whenever finite and provides an important representation of its gradient at the points of differentiability. Below we use $\overline U$ for the closure of a set $U$. 
\begin{proposition}\label{prop:bdelta}
Assume conditions \ref{A$_1$}, \ref{A$_2$}, \ref{B}, \ref{C} and \ref{E}. Fix $(\hat t,\hat x_2,\ldots \hat x_d)\in(0,T)\times\reali^{d-1}$ and assume that there exists an open bounded neighbourhood $U$ of $(\hat t,\hat x_2,\ldots \hat x_d)$ and numbers $-\infty<\underline b< \overline b<+\infty$ such that 
\begin{align}
\label{bdfinite1}&\underline b< b(t,x_2,\ldots x_d)< \overline b\quad\text{for $(t,x_2,\ldots x_d)\in \overline U$,}\\[+3pt]
\label{bdfinite2}&\partial_1(h+m)( t,x_1, x_2,\ldots  x_d)>0\quad\text{for $(t,x_2,\ldots x_d)\in \overline U$ and $x_1\in[\underline b,\overline b]$.}
\end{align} 
Then there exists $\delta_U>0$ such that $b_\delta$ is Lipschitz in $U$ for all $\delta\in(0, \delta_U]$.
Moreover, for all $k\ge 2$, and for a.e.~$(t,y)\in U$ we have
\begin{align}
\label{lip00}&\partial_k b_\delta(t,y)=-\frac{w^\circ_k(t,b_\delta(t,y),y)}{w^\circ_1(t,b_\delta(t,y),y)},\\[+6pt]
\label{lip01}&-\frac{\overline w(t,b_\delta(t,y),y)}{w^\circ_1(t,b_\delta(t,y),y)}\le \partial_t b_\delta(t,y)\le-\frac{\underline w(t,b_\delta(t,y),y)}{w^\circ_1(t,b_\delta(t,y),y)}.
\end{align}
\end{proposition}
\begin{proof}
From now on we denote 
\begin{align*}
U_b:=\{(t,x)\,:\,x_1\in(\underline b,\overline b),\,(t,x_2,\ldots x_d)\in U\}.
\end{align*}
Since $\overline U_b$ is compact and $w$ is continuous, then there exists $\delta_U>0$ sufficiently small and such that $b_\delta$ is bounded on $U$ for all $\delta\le \delta_U$, due to \eqref{bdfinite1}. With no loss of generality we then assume that $b_\delta(t,y)\in(\underline b,\overline b)$ for $(t,y)\in U$ and all $\delta\le\delta_U$.
Next we show Lipschitz regularity of $b_\delta$. 
  
For all $(t,y)\in U$ the map $x_1\mapsto w(t,x_1,y)$ is Lipschitz (Theorem \ref{thm:lip}). Then it is differentiable with $\partial_1 w(t,x_1,y)>0$, for a.e.~$x_1$ such that $(t,x_1,y)\in\cC\cap U_b$, by \eqref{monot2b} and \eqref{bdfinite2}. It follows that for all $(t,y)\in U$ the mapping $x_1\mapsto w(t,x_1,y)$ is strictly increasing on $(b(t,y), \overline b)$ and therefore a version of the implicit function theorem (see \cite{Kum80}) implies that $b_\delta$ is continuous in $U$. 

For $\eps\in\reali$ we denote $b^\eps_\delta:=b_\delta(\hat t, \hat x_2+\eps,\hat x_3,\ldots \hat x_d)$ and $b^0_\delta=b_\delta$. With no loss of generality we assume that $(\hat t,b^\eps_\delta,\hat x_2+\eps,\hat x_3 \ldots \hat x_d)$ and $(\hat t,b^0_\delta,\hat x_2,\hat x_3 \ldots \hat x_d)$ lie in $U_b$. Since $b_\delta\in C(U)$ and we are interested in small $\eps$, there is again no loss of generality in assuming that 
\[
\zeta\mapsto (\hat t, b^\eps_\delta,\hat x_2+\zeta,\hat x_3,\ldots \hat x_d),\,\,\eta\mapsto (\hat t,\eta,\hat x_2,\ldots \hat x_d)
\]
lie in a compact $K\subset U_b\cap\cC$, for $\zeta\in (0,\eps)$ and $\eta\in(b_\delta\wedge b^\eps_\delta,b_\delta\vee b^\eps_\delta)$ .

Using that $w$ is Lipschitz in $U_b$ (see Theorem \ref{thm:lip}), for $\eps\in\reali$ we have
\begin{align}\label{e1}
0=&w(\hat t, b^\eps_\delta,\hat x_2+\eps,\hat x_3,\ldots \hat x_d)-w(\hat t, b_\delta,\hat x_2,\ldots \hat x_d)\nonumber\\
=&w(\hat t, b^\eps_\delta,\hat x_2+\eps,\ldots \hat x_d)-w(\hat t, b^\eps_\delta,\hat x_2,\ldots \hat x_d)+w(\hat t, b^\eps_\delta,\hat x_2,\ldots \hat x_d)-w(\hat t, b_\delta,\hat x_2,\ldots \hat x_d)\nonumber\\
=&\int_0^\eps\partial_2 w(\hat t, b^\eps_\delta,\hat x_2+\zeta,\hat x_3,\ldots \hat x_d)d\zeta+\int^{b^\eps_\delta}_{b_\delta}\partial_1 w(\hat t,\zeta,\hat x_2,\ldots \hat x_d)d\zeta.
\end{align}
Then, using Lemma \ref{lem:wxLB} with $\cO=U_b$, we have $\partial_1w\ge c_{K,\delta}>0$ on $K$ for a suitable constant $c_{K,\delta}$ depending on $K$ and $\delta$. From the last expression in \eqref{e1} we get 
\begin{align*}
\left|\int_0^\eps\partial_2 w(\hat t, b^\eps_\delta,\hat x_2+\zeta,\hat x_3,\ldots \hat x_d)d\zeta\right|=\left|\int^{b^\eps_\delta}_{b_\delta}\partial_1 w(\hat t,\zeta,\hat x_2,\ldots \hat x_d)d\zeta\right|\ge c_{K,\delta}|b^\eps_\delta-b_\delta|
\end{align*}
and using that $|\partial_2 w|\le c'_U$ a.e.~on $U_b$ for a suitable constant $c'_U$ (see Theorem \ref{thm:lip}), we conclude 
\begin{align}\label{bdlip1}
|b^\eps_\delta-b_\delta|\le c'_U/c_{K,\delta}\,\cdot| \eps|.
\end{align}

The same argument may be repeated for the remaining variables $x_k$, $k\ge 3$, and for $t$. Then, for any $(t',x'):=(t',x'_2,\ldots x'_d)$ and $(t,x):=(t,x_2,\ldots x_d)$ belonging to a small ball in $U$, we have
\begin{align}
|b_\delta(t',x')-b_\delta(t,x)|\le c''_{\delta}(|t-t'|+\|x-x'\|_{d-1})
\end{align}
for a suitable constant $c''_{\delta}$ which depends on the small ball as well. This proves that $b_\delta$ is locally Lipschitz in $U$, hence it is differentiable almost everywhere therein. 

Next we obtain probabilistic bounds for the gradient of $b_\delta$. For $\eps>0$ we adopt the notation of Lemma \ref{lem:directional}. To simplify the exposition we set $c:=(b_\delta(t,x_2,\ldots x_d),x_2,\ldots x_d)$, so that $c^{\pm}_{k,\eps}$ and $c^{\pm}_{t,\eps}$ have the same meaning as in Lemma \ref{lem:directional} but with $b_\delta(\cdot)$ instead of the function $c(\cdot)$. Recall that $b_\delta\in (\underline b,\overline b)$ on $U$. Then, for all $k\ge 2$
\begin{align}
\delta=w(t,c)=w(t,c^-_{k,\eps})=w(t,c^+_{k,\eps})=w(t-\eps,c^-_{t,\eps})=w(t+\eps,c^+_{t,\eps}).
\end{align}
Hence by \eqref{Dx-wb}--\eqref{Dt-wb2} we obtain that if $b_\delta$ is differentiable at $(t,x_2,\ldots x_d)$, then 
\begin{align}
&\partial_k b_\delta(t,x_2,\ldots x_d)=-\frac{w^\circ_k(t,b_\delta(t,x_2,\ldots),x_2,\ldots x_d)}{w^\circ_1(t,b_\delta(t,x_2,\ldots),x_2,\ldots x_d)}\\
&-\frac{\overline w(t,b_\delta(t,x_2,\ldots),x_2,\ldots x_d)}{w^\circ_1(t,b_\delta(t,x_2,\ldots),x_2,\ldots x_d)}\le \partial_t b_\delta(t,x_2,\ldots x_d)\le -\frac{\underline w(t,b_\delta(t,x_2,\ldots),x_2,\ldots x_d)}{w^\circ_1(t,b_\delta(t,x_2,\ldots),x_2,\ldots x_d)}.
\end{align}
Since $b_\delta$ is differentiable almost everywhere in $U$ then \eqref{lip00} and \eqref{lip01} follow.  
\end{proof}

Using the bounds obtained for $\partial_t b_\delta$ and $\nabla_x b_\delta$ we can now prove the main theorems of this section. In what follows $T$ may be infinite unless stated otherwise. The first theorem uses the next condition.
\begin{description}
\item[(F) Bounds II. \namedlabel{F}{(F)}] For $(t,x)\in[0,T]\times\reali^d$ there exists $c>0$ such that 
\begin{align}
&\sum^d_{j=1}|\partial_j(h+m)(t,x)|+|\partial_t(h+m)(t,x)|\le c\,\partial_1(h+m)(t,x),
\end{align}
and if $T<+\infty$ then also 
\begin{align}
\label{eq:T00}&|h(T,x)+n(x)|+2|\partial_t f(T,x)|+\sum^d_{j=1}|\partial_j(g(x)-f(T,x))|\le c\,\partial_1(g(x)-f(T,x)).
\end{align}
\end{description}
 
Now we can state the theorem and give its proof.
\begin{theorem}[\textbf{Statement under \ref{F}}]\label{thm:F}
Assume that $d\ge2$ and conditions \ref{A$_1$}, \ref{A$_2$}, \ref{B}, \ref{C}, \ref{E} and \ref{F} hold. Assume also that the bound $\sum^d_{j=1}\|\nabla_x \mu_j(x)\|_d\le c$ holds for all $x\in\reali^d$ 
and for some $c>0$. 

If there exists $(\hat{t},\hat{x}_2,\ldots \hat{x}_d)$ and an open bounded neighbourhood $U$ of the point such that \eqref{bdfinite1} and \eqref{bdfinite2} hold then $b$ is Lipschitz on $U$. 
\end{theorem}
\begin{proof}
We provide a full proof only for $T<+\infty$ but the same arguments hold for $T=+\infty$ up to simple minor changes.

The main idea of this proof is to show that 
\begin{align}\label{eq:c0}
\|\partial_t b_\delta\|_{U,\infty}+\sum^d_{k=2}\|\partial_k b_\delta\|_{U,\infty}\le c
\end{align}
for a uniform $c>0$. Here $\|\cdot\|_{U,\infty}$ is the usual $L^\infty(U)$ norm and we are going to use the expressions for $w^\circ_k$, $\overline w$ and $\underline w$ (see \eqref{eq:wcirc-k}--\eqref{eq:wtdwn}) to find bounds in \eqref{lip00} and \eqref{lip01} (notice that the latter hold in the a.e.~sense).

Recalling \eqref{gradX}, the fact that $\partial_1\mu_j(x)=0$ for $j>1$ due to \ref{B}, and the bounds in \ref{F} it is not hard to verify that \eqref{eq:wcirc-k} gives
\begin{align}\label{gradw}
|w^\circ_k(t,x)|\le \beta_0 w^\circ_1(t,x),\qquad (t,x)\in[0,T]\times\reali^d
\end{align}
for all $k=2,\ldots d$, and a suitable $\beta_0>0$ which is independent of $t,x$ and $k$. 
Similarly, using the bounds \ref{F} in \eqref{eq:wtup}--\eqref{eq:wtdwn} we can find $\beta_1>0$ such that  
\begin{align}\label{wt}
\max\{|\underline w|,|\overline w|\}(t,x)\le \beta_1 w^\circ_1(t,x),\qquad (t,x)\in[0,T]\times\reali^d.
\end{align}

Now we argue as in the proof of Theorem \ref{thm:lip-d1} and since $|b|<+\infty$ on $U$ we can find $\delta_U>0$ sufficiently small and such that $w(t,\,\cdot\,,x_2,\ldots x_d)=\delta$ has a solution $x_1=b_\delta(t,x_2,\ldots x_d)$, which is finite in $U$ for all $\delta\le \delta_U$. Then $b_\delta$ is Lipschitz in $U$ by Proposition \ref{prop:bdelta}. Moreover from \eqref{lip00}, \eqref{lip01} and \eqref{wt} we obtain, for a.e.~$(t,y)\in U$
\begin{align}\label{lip02}
|\partial_t b_\delta(t,y)|\le \frac{\max\{|\underline w|,|\overline w|\}(t,b_\delta(t,y),y)}{ w^\circ_1(t,b_\delta(t,y),y)} \le \beta_1
\end{align}
and from \eqref{gradw}
\begin{align}\label{lip03}
|\partial_k b_\delta(t,y)|=\frac{|w^\circ_k(t,b_\delta(t,y),y)|}{w^\circ_1(t,b_\delta(t,y),y)} \le \beta_0,\quad k=2,\ldots d.
\end{align}
As in the proof of Theorem \ref{thm:lip-d1} (see \eqref{lim-p1}) we have pointwise convergence $b_\delta\downarrow b$ as $\delta\to 0$ and therefore, by using again Ascoli-Arzel\`a's theorem, we conclude that $b_\delta\to b$ uniformly on $U$. Hence $b$ is Lipschitz on $U$. 
\end{proof}

For the next theorem we are going to use a different technical condition. 
\begin{description}
\item[(G) Bounds III. \namedlabel{G}{(G)}] There exists $c_1,c_2>0$ such that 
\begin{align}
\label{G1}&\partial_1(h+m)(t,x)\ge c_1\\
\label{G3}&\sum^d_{j=1}|\partial_j(h+m)(t,x)|+|\partial_t(h+m)(t,x)|\le c_2\,\left[1+\partial_1(h+m)(t,x)\right]
\end{align}
for $(t,x)\in[0,T]\times\reali^d$. 
Moreover, if $T<+\infty$, at least one of the two conditions below holds:
\begin{itemize}
\item[(a)] for $(t,x)\in[0,T]\times\reali^d$ it holds
\begin{align}
&|h(T,x)+n(x)|+2|\partial_t f(T,x)|\nonumber\\
&\hspace{+20pt}+\sum^d_{j=1}|\partial_j(g(x)-f(T,x))|\le c_2\,\left[1+\partial_1(g(x)-f(T,x))\right];
\end{align}
\item[(b)] for some $p\ge 1$ and for all $(t,x)\in[0,T]\times\reali^d$ it holds
\begin{align}
\label{G4}&\sum^d_{j=1}|\partial_j(g(x)-f(T,x))|\le c_2\,\left[1+\partial_1(g(x)-f(T,x))\right],\\
\label{polyn}&|h(T,x)+n(x)|+|\partial_tf(t,x)|\le c_2(1+\|x\|_d^p).
\end{align}
\end{itemize}
\end{description}

Now we can state the theorem and provide its proof.
\begin{theorem}[\textbf{Statement under \ref{G}}]\label{thm:G}
Assume that $d\ge2$ and conditions \ref{A$_1$}, \ref{A$_2$}, \ref{B}, \ref{C}, \ref{E} and \ref{G} hold. Let $\sum^d_{j=1}\|\nabla_x \mu_j(x)\|_d\le c$ hold true for all $x\in\reali^d$ and a given $c>0$ and, if $T<+\infty$ and only (b) of \ref{G} holds, assume also $\|\mu(x)\|_d\le c$ for $x\in\reali^d$. 

If there exists $(\hat{t},\hat{x}_2,\ldots \hat{x}_d)$ and an open bounded neighbourhood $U$ of the point such that \eqref{bdfinite1} holds, then $b$ is Lipschitz on $U$. 
\end{theorem}
\begin{proof}
We provide a full proof only for $T<+\infty$ but the same arguments hold for $T=+\infty$ up to simple minor changes.

The idea of the proof is the same as in the previous theorem.
Recalling \eqref{gradX}, the fact that $\partial_1\mu_j(x)=0$ for $j>1$ due to \ref{B}, and the bounds in \ref{G} it is not hard to verify that \eqref{eq:wcirc-k} gives
\begin{align}\label{gradw2}
|w^\circ_k(t,x)|\le &\beta_0\left(\media_{t,x}(\tau_*-t)+\prob_{t,x}(\tau_*=T) + w^\circ_1(t,x)\right)\\
\le&\beta_0(1+(T-t)^{-1})\media_{t,x}(\tau_*-t)+\beta_0 w^\circ_1(t,x)\nonumber
,\qquad (t,x)\in[0,T]\times\reali^d
\end{align}
for all $k=2,\ldots d$, and a suitable $\beta_0>0$ which is independent of $t,x$ and $k$. Notice that the second inequality is just an application of Markov inequality (see \eqref{eq:markov}).

For the bounds on $\underline w$ and $\overline w$ we treat separately the case in which condition (a) of \ref{G} holds and (b) of \ref{G} holds. Starting with the former and recalling \eqref{eq:wtup}--\eqref{eq:wtdwn} it is not hard to show that there exists a constant $\beta_1>0$ such that for all $(t,x)\in[0,T)\times\reali^d$
\begin{align}\label{wwt2}
\max&\{|\underline w|,|\overline w|\}(t,x)\nonumber\\
\le& \media_{t,x}\left[\int_t^{\tau_*}|\partial_t(h+m)(s,X_s)|ds+\mathds{1}_{\{\tau_*=T\}}\left|(h+\partial_t f)(T,X_{T})+n(X_T)\right|\right]\nonumber\\
\le& \beta_1\left(\media_{t,x}(\tau_*-t)+\prob_{t,x}(\tau_*=T) + w^\circ_1(t,x)\right)\nonumber\\
\le& \beta_1(1+(T-t)^{-1})\media_{t,x}(\tau_*-t)+\beta_1 w^\circ_1(t,x).
\end{align}
Under condition (b) instead we use Lemma \ref{lem:growth} to get 
\begin{align*}
\media_{t,x}\left[\mathds{1}_{\{\tau_*=T\}}\left|(h+\partial_t f)(T,X_{T})+n(X_T)\right|\right]
\le \beta'_1(1+\|x\|^p_d)(T-t)^{-1}\media_{t,x}(\tau_*-t)
\end{align*}
for some $\beta'_1>0$, while the estimate for the running cost $\partial_t(h+m)$ is the same as in \eqref{wwt2}. Therefore we can find again $\beta_1>0$ such that, for $(t,x)\in[0,T)\times\reali^d$  
\begin{align}\label{wwt3}
\max&\{|\underline w|,|\overline w|\}(t,x)\nonumber\\
\le& \beta_1\left[1+(1+\|x\|^p_d)(T-t)^{-1}\right]\media_{t,x}(\tau_*-t)+\beta_1w^\circ_1(t,x).
\end{align}

We notice that thanks to \ref{G} we also have
\begin{align*}
w^\circ_1(t,x)\ge c\,\media_{t,x}(\tau_*-t), \quad(t,x)\in[0,T]\times\reali^d.
\end{align*}

To show the Lipschitz property let $b_\delta$ be the $\delta$-level set of $w$ and let us find bounds for \eqref{lip00} and \eqref{lip01}. In particular for a.e.~$(t,y)\in[0,T)\times\reali^{d-1}$ we estimate
\begin{align}
|\partial_t b_\delta(t,y)|\le &\frac{\max\{|\underline w|,|\overline w|\}(t,b_\delta(t,y),y)}{w^\circ_1(t,b_\delta(t,y),y)}\nonumber \\
\le& \beta_1+\frac{\beta_1}{c}\left[1+(1+\|(b_\delta,y)\|^p_d)(T-t)^{-1}\right],
\end{align}
with $(b_\delta,y):=(b_\delta(t,y),y)$ for simplicity of notation, and 
\begin{align}
|\partial_k b_\delta(t,y)|=\frac{|w^\circ_k(t,b_\delta(t,y),y)|}{w^\circ_1(t,b_\delta(t,y),y)}\le \beta_0+\frac{\beta_0}{c}(1+(T-t)^{-1})
\end{align}
for $k=2,\ldots d$. The rest of the proof follows by letting $\delta\downarrow 0$ and using Ascoli-Arzel\`a's theorem as in the proof of Theorem \ref{thm:lip-d1}.
\end{proof}
\begin{remark}\label{rem:comp}
The last two theorems above work under weaker technical conditions on $f$, $g$ and $h$ than those imposed in \cite{SS91}. It is worth drawing a precise parallel between the two contributions and it is important to notice that some inequalities are reversed just because \cite{SS91} considers a problem of optimal stopping with minimisation of costs, and in which the stopping set lies above the continuation set.

As for the notation, \cite{SS91} takes a state-independent obstacle, i.e.~$f(t,x)=f(t)$, and a different ordering of the space coordinates. Indeed our $\partial_k(h+m)(t,x)$ should be associated to $h_{n,n+1-k}(t,x)$ in \cite{SS91} for $k=1,\ldots n$. Similarly our $\partial_k (g(x)-f(T,x))$ corresponds to $g_{n,n+1-k}(x)$ of \cite{SS91}.  The setting we adopted in Theorem \ref{thm:G}, with \ref{G} using the specifications in (b), is more general than the setting in \cite{SS91}; in particular conditions in (2.2) of \cite{SS91} imply our \eqref{G1}--\eqref{G3} and \eqref{G4}. The polynomial growth \eqref{polyn} is the same as in \cite{SS91} and the requirement $g_n(x)\le f(0)$ of \cite{SS91} corresponds to our $g(x)\ge f(T,x)$.

Finally we notice that results in \cite{SS91} are obtained for $\mu(x)\equiv 0$ and $\sigma=\text{diag}\,1$ in \eqref{X}.
\end{remark}

\section{Some examples and further extensions}\label{sec:examples}

Here we illustrate a couple of applications of our results to problems studied in the literature on stochastic control \cite{Ch-Haus09}, \cite{DeAFF}. The Lipschitz regularity of the free boundary in such problems is new and was not discovered in \cite{Ch-Haus09} and \cite{DeAFF}. In all the examples below it is not difficult to check that the standing assumptions \eqref{integr}, \eqref{integr2}, \eqref{grad-L2} and Assumption \ref{ass:fgh} hold.
\vspace{+5pt}

\noindent \textbf{Example 1}. 
Here we consider a problem of optimal stopping arising in connection with one of irreversible investment (see \cite{Ch-Haus09}), under a Cobb-Douglas type production function. The state space is $[0,T]\times\reali$ and the optimisation problem reads (see (3.15) and Section 4 in \cite{Ch-Haus09})
\begin{align}\label{P-CH}
v(t,x)=\sup_{0\le \tau\le T-t}\media\left[-\int_0^\tau e^{-rs-(1-\alpha)X^x_s}ds-c_1e^{-r\tau}\mathds{1}_{\{\tau<T-t\}}-c_2e^{-r\tau}\mathds{1}_{\{\tau=T-t\}}\right]
\end{align}
where $r>0$, $\alpha\in(0,1)$, $c_1\ge c_2\ge 0$ and 
\begin{align*}
X^x_t=x+\mu t+\sigma B_t, \qquad x\in\reali.
\end{align*}
Notice that, due to discounting, in this example we must replace the infinitesimal generator $\cL$ by $\cL-r$, which corresponds to the diffusion $X$ killed at the constant rate $r$. In this setting we have 
\begin{align*}
&h(x)=-e^{-(1-\alpha)x},\quad f(x)=-c_1,\quad g(x)=-c_2,\quad m(x)=rc_1,\quad n(x)=rc_2,\\
&\partial_t(h+m)=0,\quad \partial_x(h+m)(x)=(1-\alpha)e^{-(1-\alpha)x}.
\end{align*}
Here we want to use Theorem \ref{thm:lip-d1} and we start by noticing that conditions \ref{A$_1$}, \ref{A$_2$}, \ref{C} and \ref{D} hold. Moreover, the curve $\gamma$ (see \eqref{theta}) is simply given by 
\begin{align*}
\gamma(t)=\gamma=(1-\alpha)^{-1}\ln (1/rc_1),
\end{align*}
{\color{blue}so that $(i)$ of Theorem \ref{thm:lip-d1} holds. 
%Finiteness of $b(t)$ for $t\in[0,T)$ was shown in \cite{Ch-Haus09} and therefore $(i)$ of our Theorem \ref{thm:lip-d1} is satisfied. 
As for $(ii)$ it is immediate to check that for any $x_0>\gamma$ one has $\partial_x(h+m)\ge (1-\alpha)e^{-(1-\alpha)x_0}$ for $x\le x_0$.} It only remains to check the requirements of Corollary \ref{cor:bounds}. If $c_2=c_1$ then $(i)$ in the latter corollary holds and therefore we have
\begin{proposition}
For $c_1=c_2$ Theorem \ref{thm:lip-d1} is true for problem \eqref{P-CH}.
\end{proposition}

We notice that $(ii)$ in Corollary \ref{cor:bounds} is too strong in this setting and can never be verified. Moreover, in \cite{Ch-Haus09} they consider $c_1>0$ and $c_2=0$, so that the assumptions of Corollary \ref{cor:bounds} do not hold. However, as already mentioned, the key point in our method is the probabilistic representation for the bounds of $\partial_t v$ and $\partial_x v$. In particular, the explicit nature of problem \eqref{P-CH} allows us to refine \eqref{downvt} and this turns out to be sufficient to prove Lipschitz regularity of $b$. In what follows we achieve this task.

\begin{proposition}
If $c_1>0$ and $c_2=0$ then $b$ is Lipschitz on $[0,T)$.
\end{proposition}
\begin{proof}
First notice that \eqref{grad-v} and \eqref{upvt} give
\begin{align}\label{eq:gab2}
\partial_x v(t,x)=(1-\alpha)\media\left[\int_0^{\tau_*}e^{-rs -(1-\alpha)X^x_s}ds\right]
\end{align}
and
\begin{align}\label{eq:gab3}
\overline v(t,x)=e^{-r(T-t)}\media\left[\mathds{1}_{\{\tau_*=T-t\}}e^{-(1-\alpha)X^x_{T-t}}\right].
\end{align}
For the lower bound of $\partial_t v$ we follow the proof of Theorem \ref{thm:lip} up to \eqref{vt0}, which now reads
\begin{align*}
v(t+\eps,x)-v(t,x)\ge& -\media\left[\mathds{1}_{\{\tau>T-t-\eps\}}\left(e^{-r\tau}v(t+\tau,X^x_\tau)+\int^\tau_{T-t-\eps}e^{-rs}h(X^x_s)ds\right)\right]\\[+4pt]
=&-\media\left[\mathds{1}_{\{\tau>T-t-\eps\}}e^{-r(T-t-\eps)}v(T-\eps,X^x_{T-t-\eps})\right]\ge 0
\end{align*}
since $v\le 0$ by \eqref{P-CH}.
Since $\partial_t v$ exists at all points of $\cC$, the above gives $\partial_t v \ge 0$.

The latter is useful in estimating $b'_\delta$ in \eqref{eq:gab1}. In fact we immediately obtain (noticing that here $\partial_t v=\partial_t w$ and $\partial_x v=\partial_x w$)
\begin{align}
0\ge b'_\delta(t)\ge- \frac{\overline v(t,b_\delta(t))}{\partial_x v(t,b_\delta(t))}.
\end{align}
To estimate the right-hand side of the above we observe that
\begin{align*}
e^{-(1-\alpha)X^x_t}=\frac{d \widetilde \prob}{d \prob}\Big|_{\cF_t}\,\vartheta(t,x)
\end{align*}
with
\begin{align*}
\frac{d \widetilde \prob}{d \prob}\Big|_{\cF_t}:=e^{(\alpha-1)\sigma B_t-\tfrac{1}{2}(\alpha-1)^2\sigma^2 t}\quad\text{and}\quad \vartheta(t,x):=e^{(\alpha-1)(x+\mu t+\tfrac{1}{2}(\alpha-1)\sigma^2 t)}.
\end{align*} 
Using the new probability measure in \eqref{eq:gab2} and \eqref{eq:gab3} we get
\begin{align}
0\ge b'_\delta(t)\ge -\frac{\vartheta(T-t,x)\pprob(\tau_*=T-t)}{(1-\alpha)\mmedia\left[\int_0^{\tau_*}e^{-rs}\vartheta(s,x)ds\right]}\ge -\frac{(T-t)^{-1}\vartheta(T-t,x)\mmedia[\tau_*]}{(1-\alpha)\underline \vartheta(t,x)\, \mmedia[\tau_*]}
\end{align}
by using Markov inequality in the numerator and setting $\underline \vartheta (t,x):=\inf_{0\le s\le T-t}e^{-rs}\theta(s,x)$.
Hence we conclude
\begin{align*}
0\ge b'_\delta(t)\ge  -\frac{(T-t)^{-1}\vartheta(T-t,x)}{(1-\alpha)\underline \vartheta(t,x)}
\end{align*}
\end{proof}
\vspace{+5pt}

\noindent \textbf{Example 2}.
Now we consider a multi-dimensional problem which arises in connection with irreversible investment with stochastic prices and demand (see \cite{DeAFF}). The state variables are $(t,x,y,z)\in[0,T]\times\reali^3$ and we treat separately the case $T<+\infty$ and $T=+\infty$. The problem is degenerate because there is no dynamics in the $z$ direction (however a dynamic with no Brownian part could be included without altering our analysis). This can also be seen as a family of problems depending on a parameter $z$. The \emph{most interesting feature} is that we can prove Lipschitz continuity of the boundary also with respect to the parameter $z$. 

The optimisation problem reads
\begin{align}\label{P-D}
v(t,x,y,z)=\sup_{0\le \tau\le T-t}\media\left[\int_0^{\tau}e^{-r s}(z-X^x_t)dt-e^{-r \tau}Y^y_\tau\right]
\end{align}
where $r>0$. It is worth mentioning that the running cost above corresponds to taking $c(x,z)=\tfrac{1}{2}(x-z)^2$ in the control problem studied by \cite{DeAFF}. As discussed in the latter paper, this choice is very natural in that context (see Remark 2.5 in \cite{DeAFF}). Notice also that the form of the payoff prevents a dimensionality reduction, hence also the infinite horizon case $T=+\infty$ is truly 2-dimensional and parameter dependent. 

A typical application of this model is for electricity generation in presence of renewable sources. Here $X$ is associated to a stochastic demand net of generation from renewables, and $Y$ to the stochastic spot price of electricity. The variable $Z$ is related to the level of production capacity from conventional generation. In this model both demand and price can take negative values, which is consistent with real market observations. We consider three different cases to illustrate our methodology in full.
\vspace{+4pt}

\noindent \emph{Example 2-(a)}. Let us start with a finite-horizon problem with simple dynamics. Let $T<+\infty$ and 
\begin{align}\label{dynamics0}
&X^x_t=x+\alpha t+\beta B_t \quad \text{and}\quad Y^y_t=y+\mu t+\sigma W_t
\end{align}
where $\alpha,\beta,\sigma,\mu$ are constants while $B$ and $W$ are (possibly correlated) Brownian motions. 

Again we are in presence of a killing at a rate $r$ and therefore we use $\cL-r$ instead of $\cL$. We have
\begin{align*}
&h(x,z)=z-x,\quad f(y)=g(y)=-y,\quad m(y)=n(y)=ry-\mu,\\
&\partial_t(h+m)=0,\quad \partial_x(h+m)(x,y,z)=-1,\quad \partial_z(h+m)(x,y,z)=1,\quad\partial_y(h+m)(x,y,z)=r.
\end{align*}
Considering $\partial_1=\partial_y$, it is immediate to check that Proposition \ref{b:existence} holds and we have 
\begin{align*}
\cS=\{(t,x,y,z)\,:\,y\le b(t,x,z)\}.
\end{align*}
Finiteness of the boundary was proved in \cite{DeAFF} for the infinite horizon case, and therefore holds as well for $T<+\infty$. Moreover all assumptions of Theorem \ref{thm:G} hold under condition (b) of \ref{G}. Hence we have
\begin{proposition}
For $T<+\infty$ and with the dynamics \eqref{dynamics0}, Theorem \ref{thm:G} is true for problem \eqref{P-D}.
\end{proposition}
\vspace{+5pt}

\noindent \emph{Example 2-(b)}. Let $T=+\infty$ and the dynamic of $(X,Y)$ be
\begin{align}\label{dynamics1}
&X^x_t=x+\alpha\int_0^t(\zeta-X^x_s)ds+\beta B_t \quad \text{and}\quad Y^y_t=y+\mu t+\sigma W_t
\end{align}
where $\alpha,\beta,\sigma,\mu,\zeta$ are constants while $B$ and $W$ are (possibly correlated) Brownian motions. In this setting we assume a mean reverting dynamic for the demand. For the finiteness of \eqref{P-D} and to guarantee \eqref{integr2} we pick $r>\alpha$. We have
\begin{align*}
&h(x,z)=z-x,\quad f(y)=g(y)=-y,\quad m(y)=n(y)=ry-\mu,\\
&\partial_t(h+m)=0,\quad \partial_x(h+m)(x,y,z)=-1,\quad \partial_z(h+m)(x,y,z)=1,\quad\partial_y(h+m)(x,y,z)=r.
\end{align*}
As in case (a) above, taking $\partial_1=\partial_y$, we see immediately that Proposition \ref{b:existence} holds and we have 
\begin{align}\label{ex:S}
\cS=\{(x,y,z)\,:\,y\le b(x,z)\}.
\end{align} 
Now Theorems \ref{thm:F} and \ref{thm:G} hold, because both conditions \ref{F} and \ref{G} hold due to infinite horizon (hence the unbounded drift for $X$ is admissible). So we have 
\begin{proposition}
For $T=+\infty$ and with the dynamics \eqref{dynamics1}, Theorems \ref{thm:F} and \ref{thm:G} are true for problem \eqref{P-D}.
\end{proposition}
\vspace{+5pt}

\noindent \emph{Example 2-(c)}. Here we want to consider non-negative prices and this will require to adapt our methods as we did in Example 1. Let $T=+\infty$ and the dynamic of $(X,Y)$ be
\begin{align}\label{dynamics2}
&X^x_t=x+\alpha\int_0^t(\zeta-X^x_s)ds+\beta B_t \quad \text{and}\quad Y^y_t=y+\mu\int_0^t Y^y_s ds+\sigma\int_0^tY^y_s dW_s,
\end{align}
where $\alpha,\beta,\sigma,\mu,\zeta$ are constants while $B$ and $W$ are independent Brownian motions. We take $r>\alpha\vee \mu$ to guarantee finiteness of the value function and \eqref{integr2}.

To fit in the framework of \eqref{X} we must consider the new state variable $\pi:=\ln y$ so that a new process $\Pi_t$ can be defined as
\begin{align*}
Y^y_t=\exp\{\pi+(\mu-\tfrac{\sigma^2}{2})t+\sigma W_t\}=:\exp\{\Pi^\pi_t\}
\end{align*}
In the state variables $(x,\pi,z)$ we have
\begin{align*}
&h(x,z)=z-x,\quad f(\pi)=g(\pi)=-e^\pi,\quad m(\pi)=n(\pi)=(r-\mu)e^\pi,\quad \partial_t(h+m)=0,\\
&\partial_x(h+m)(x,\pi,z)=-1,\quad \partial_z(h+m)(x,\pi,z)=1,\quad\partial_\pi(h+m)(x,\pi,z)=(r-\mu)e^\pi.
\end{align*}
Now the stopping set reads $\cS	=\{(x,\pi,z):\pi\le b(x,z)\}$ and the boundary is finite. It was shown in \cite{DeAFF} that $w:=v+e^\pi$ is differentiable in $x$ and $\pi$ inside the continuation region. Moreover, with a slight abuse of notation, we also define $\partial_z w:=w^\circ_3$ (see \eqref{eq:wcirc-k}).
\begin{proposition}
For $T=+\infty$ and with the dynamics \eqref{dynamics2} we have $b$ Lipschitz on $\reali^2$.
\end{proposition}
\begin{proof}
For $(x,\pi,z)\in\reali^3$ let $\tau_*=\tau_*(x,\pi,z)$ for simplicity.
From \eqref{grad-v} and \eqref{eq:wcirc-k} we get 
\begin{align}
\label{tt00}&\partial_x w(x,\pi,z)=-\media\left[\int_0^{\tau_*}e^{-(r+\alpha)t}dt\right],\quad \partial_\pi w(x,\pi,z)=(r-\mu)\mmedia\left[\int_0^{\tau_*}e^{-(r-\mu)t}dt\right]\\
\label{tt01}&\partial_z w(x,\pi,z)=\media\left[\int_0^{\tau_*}e^{-r t}dt\right]
\end{align}
where the measure $\pprob$ is defined by
\begin{align}
\frac{d\,\pprob}{d\,\prob}\Big|_{\cF_t}:=e^{\sigma W_t-\tfrac{\sigma^2}{2}t},\qquad t\ge 0.
\end{align}
The crucial observation now is that, while the dynamic of $X$ is unaffected by the change of measure (due to the independence of $B$ and $W$) the dynamic of $\Pi$ under $\pprob$ becomes 
\begin{align*}
\Pi_t^\pi=\pi+(\mu+\tfrac{\sigma^2}{2})t+\sigma\widetilde W_t
\end{align*}
where $\widetilde W_t=W_t-\sigma t$ is the new Brownian motion under $\pprob$. 

Now, if on $(\Omega,\cF,\prob)$ we define
\begin{align*}
\widetilde \Pi^\pi_t=\pi+(\mu+\tfrac{\sigma^2}{2})t+\sigma W_s
\end{align*}
and $\widetilde\tau_*:=\inf\{t\ge0\,:\, \widetilde \Pi^\pi_t\le b(X^x_t,z)\}$,
then we immediately see that 
\begin{align*}
\mathsf{Law}(X^x,\Pi^\pi|\pprob)=\mathsf{Law}(X^x,\widetilde \Pi^\pi|\prob)\quad\text{and}\quad\mathsf{Law} (\tau_*|\pprob)=\mathsf{Law} (\widetilde \tau_*|\prob).
\end{align*} 
Moreover by comparison principles for SDEs we have that $\widetilde \Pi^\pi_t\ge \Pi^\pi_t$, $\prob$-a.s.~for all $t\ge0$ and therefore we get $\widetilde \tau_*\ge \tau_*$, $\prob$-a.s., and 
\begin{align}\label{tt02}
\mmedia\left[\int_0^{\tau_*}e^{-(r-\mu)t}dt\right]=\media\left[\int_0^{\widetilde \tau_*}e^{-(r-\mu)t}dt\right]\ge \media\left[\int_0^{\tau_*}e^{-(r-\mu)t}dt\right]
\end{align}

Using \eqref{tt00}--\eqref{tt02} and setting $\tau_*=\tau_*(x,b_\delta(x,z),z)$ for simplicity, we find a uniform bound for $\nabla b_\delta$ for any $\delta>0$, that is
\begin{align*}
&|\partial_z b_\delta(x,z)|=\left|\frac{\partial_z w (x,b_\delta(x,z),z)}{\partial_\pi w(x,b_\delta(x,z),z)}\right|\le (r-\mu)^{-1}\frac{\media\left[\int_0^{\tau_*}e^{-r t}dt\right]}{\media\left[\int_0^{\tau_*}e^{-(r-\mu)t}dt\right]}\le \frac{1}{r-\mu}\\[+4pt]
&|\partial_x b_\delta(x,z)|=\left|\frac{\partial_x w (x,b_\delta(x,z),z)}{\partial_\pi w(x,b_\delta(x,z),z)}\right|\le (r-\mu)^{-1}\frac{\media\left[\int_0^{\tau_*}e^{-(r+\alpha) t}dt\right]}{\media\left[\int_0^{\tau_*}e^{-(r-\mu)t}dt\right]}\le \frac{1}{r-\mu}.
\end{align*}
Thus taking $\delta\to0$ we find that $b$ is Lipschitz as claimed.

\end{proof}
\vspace{+7pt}

%%%%%%%%%%%%%%%%%%%%%%%%%%%%%%%%%%%
\noindent\textbf{Acknowledgements}:\\
This work was financially supported by Sapienza University of Rome, research project ``\emph{Modelli stocastici a tempo continuo per le scelte previdenziali}'', grant no.~C26A14HXBR. T.~De Angelis was also partially supported by EPSRC grant EP/R021201/1. We thank G.~Peskir for helpful comments about general theory of optimal stopping. Finally, we thank the Associate Editor and an anonymous Referee for pertinent comments that helped improve the exposition of the paper.
%%%%%%%%%%%%%%%%%%%%%%%%%%%%%%%%%%%

\end{document}